\documentclass[11pt]{amsart}
\usepackage{tikz}
\usetikzlibrary{arrows.meta}
\usepackage{esint}
\usepackage{verbatim}
\usepackage{graphicx}
\usepackage{subfigure}
\usepackage{amsfonts}
\usepackage{amssymb}
\usepackage{amsmath}
\usepackage{amsthm}
\usepackage{latexsym}
\usepackage{amscd}
\usepackage{mathrsfs}
\usepackage{verbatim}
\usepackage{mathrsfs}
\usepackage{enumitem}
\usepackage{braket}
\usepackage{float}
\usepackage[margin=1.3in]{geometry}
\usepackage{comment}
\usepackage{bm}
\usepackage{xcolor}
\usepackage{imakeidx}
\makeindex[title=Index of Notation]

\usepackage{hyperref}
\hypersetup{
	colorlinks=true,
	linkcolor=black,
	urlcolor=black,
	citecolor=black
}
\usepackage{caption}
\usepackage{array}

\usepackage{amssymb}
\usepackage{pdfpages}
\usepackage{verbatim}

\usepackage{appendix}
\usepackage{xspace}

\newtheorem{theorem}{Theorem}[section]

\newtheorem{lemma}[theorem]{Lemma}
\newtheorem{proposition}[theorem]{Proposition}
\newtheorem{corollary}[theorem]{Corollary}
\theoremstyle{definition}
\newtheorem{definition}[theorem]{Definition}

\theoremstyle{remark}
\newtheorem{remark}[theorem]{Remark}

\makeatletter

\newcommand{\Rmnum}[1]{\expandafter\@slowromancap\romannumeral #1@}
\makeatother

\newcommand{\ID}{\mathbb{D}}

\newcommand{\IN}{\mathbb{N}}
\newcommand{\IR}{\mathbb{R}}

\numberwithin{equation}{section}

\makeatletter
\@namedef{subjclassname@2020}{\textup{2020} Mathematics Subject Classification}
\makeatother


\begin{document}
	\title[Covariant  Riesz transforms ]{The Covariant Riesz Transforms on Riemannian Manifolds}
	
\author{Yongheng Han}
\address{School of Mathematical Science, University of Science and Technology of China,  Hefei City,  Anhui Province 230026}

\email{hyh2804@mail.ustc.edu.cn}

\author{Bing Wang}
\address{ Institute of Geometry and Physics, and School of Mathematical Sciences, University of Science and Technology of China, Hefei 230026, China; Hefei National Laboratory, Hefei 230088, China}

\email{topspin@ustc.edu.cn}

	\date{\today}

	
	\begin{abstract}
		We establish the $L^p$-boundedness of the local covariant Riesz transform for differential forms on manifold $M$ with bounded $\|Rm\|$. Let $\Delta_j$ be the Hodge Laplace operator on $j$-forms. For any $p \in (1, \infty)$ and $\kappa>\kappa_0$, we show that the operator $\nabla (\Delta_j + \kappa)^{-1/2}$ is bounded on $L^p(M)$. Consequently, we obtain Calderón-Zygmund estimates for manifolds with bounded Riemannian curvature.
	\end{abstract}

	\maketitle
	\tableofcontents
	\section{Introduction}
	
	
	 In the Euclidean space $\mathbb{R}^n$, the Riesz transforms are defined as the collection of operators $R_j = \partial_j (-\Delta)^{-1/2}$. The Euclidean theory culminated in the Calderón-Zygmund framework, which characterizes Riesz transforms as singular integral operators. Within this setting, $R_j$ are identified as convolution operators with kernels $K(x) = c_n x_j |x|^{-(n+1)}$, whose $L^p$-boundedness is established via Fourier-analytic methods and real-variable theory. 
	
	The investigation of Riesz transforms on general Riemannian manifolds was initiated by Robert S. Strichartz. Extensive research has been conducted on Riesz transforms in the setting of manifolds, yielding a substantial body of literature and significant theoretical advancements. For the Riesz transform of functions, Strichartz \cite{Str83} proved the $L^p$ boundedness of the Riesz transform $\nabla(-\Delta)^{-1/2}$ for rank-one symmetric spaces (including hyperbolic space $\mathbb{H}^n$). He bypassed the heat kernel difficulties by using the wave equation method and spectral theory, utilizing the fact that the spectrum of the Laplacian on these spaces is bounded away from zero (the ``spectral gap''). Bakry \cite{Bak85} used the $\Gamma_2$-calculus to prove that on manifolds with non-negative Ricci curvature, the Riesz transform is $L^p$-bounded. Li \cite{Li91} showed that the Riesz transform on a complete manifold with non-negative Ricci curvature is of weak type $(1,1)$. Furthermore, Coulhon and Duong \cite{CD99} provided positive results for $1 < p \leq 2$ under the doubling volume property and optimal on-diagonal heat kernel estimates. For the case $p > 2$, we refer to \cite{Aus04, CD01, CD03}.
	
	 Let $\Delta_j$ be the Hodge Laplace operator on $j$-forms, $e^{-t\Delta_j}$ the associated heat semigroup, and $H_j(x,y,t)$ the corresponding heat kernel on $M$. In \cite{Bak85}, Bakry demonstrated 
	\begin{theorem}\label{Bak}\cite{Bak85}\label{thm1.6}
		Assume that $\|Rm\|\leq \Lambda_0$, then there exists a $\kappa_0=\kappa_0(\Lambda_0,n)>0$ such that for all $\kappa\geq \kappa_0$ and $j=\{0,\cdots,n\}$ the operators $d_j(\Delta_j+\kappa )^{-1/2}$ and $d_{j-1}^*(\Delta_j +\kappa)^{-1/2}$ are weak $(1,1)$. For every $p\in (1,+\infty)$, one has
		\begin{equation*}
			\bigg\|d_j(\Delta_j+\kappa)^{-1/2}\bigg\|_{p,p}<\infty,\quad \bigg\|d^*_{j-1}(\Delta_j+\kappa)^{-1/2}\bigg\|_{p,p}<\infty,
		\end{equation*}
		with norm bounds depending on $n,p,\Lambda_0, \kappa$. Here, $d_j$ denotes the exterior derivative on $j$-form, $d^*_{j-1}$ its formal adjoint.
	\end{theorem}
     If $j=0$, then $d_j(\Delta_j+\kappa)^{-1/2}=\nabla (-\Delta+\kappa)^{-1/2}$.  Theorem \ref{thm1.6} can be viewed as  a generalization of the Riesz transform on the functions.
     
     Let $\nabla$ denote the Riemannian gradient.
     The covariant Riesz transform of forms on $M$ is defined by 
	\begin{equation*}
		\mathcal{R}_j := \nabla \Delta_j^{-1/2} = \Gamma\left(\frac{1}{2}\right)^{-1} \int_0^\infty t^{-1/2} \nabla e^{-t\Delta_j} dt.
	\end{equation*}
	We can also define the local version of the covariant Riesz transform by
	\begin{equation*}
		\begin{split}
			\nabla (\Delta_j+\kappa)^{-1/2} = \Gamma\left(\frac{1}{2}\right)^{-1} \int_0^\infty t^{-1/2} e^{-\kappa t}\nabla e^{-t\Delta_j} dt,
		\end{split}
	\end{equation*}
	for any $\kappa\geq 0$. Bakry \cite{Bak85} established the boundedness of local version of the covariant Riesz transform for Einstein manifolds with bounded curvature. In \cite{DT01},  Driver-Thalmaier use martingale methods to give Bismut type derivative formulas for differentials and co-differentials of heat semigroups on forms, and more generally for sections of vector bundles. Using their formula, Thalmaier-Wang \cite{TW04} obtained derivative estimates for various heat semigroups on Riemannian vector bundles. As an
application, they established the weak (1, 1) property for a class of Riesz transforms on a vector bundle(e.g., differential forms). In \cite[Proposition 4.18]{GP15}, G\"uneysu-Pigola proved the boundedness of the covariant Riesz transform for $1<p<2$ under some assumptions of curvature and volume. Baumgarth, Devyver, and G\"{u}neysu formulated a conjecture regarding the boundedness of the local covariant Riesz transform and proved it for $1 < p \leq 2$. 
	
	\begin{theorem}\cite{BDG23}\label{BDG23}\label{thm1.1}
		Let $M$ be an $n$-manifold with $\|Rm\|+\|\nabla Rm\| \leq \Lambda_0$. Then for every $p \in (1, 2]$, there exists a constant $\kappa_0 = \kappa_0(n,\Lambda_0, p)$ such that for $\kappa \geq \kappa_0$ and $j \in \{0, \dots, n\}$, one has $\|\nabla (\Delta_j + \kappa)^{-1/2}\|_{p,p} < \infty$, with a norm depending only on $n, p, \Lambda_0$ and $\kappa$.
	\end{theorem}

	Inspired by the work of Wang \cite{Wan03}, Li-Wang \cite{LW06}, and Pigola \cite{Pig20}, we establish the following:
	
	\begin{theorem}[\textbf{Main result}]
    \label{thm1.4}
		Let $M$ be an $n$-manifold with $\|Rm\| \leq \Lambda_0$. Then for every $p \in (1, \infty)$, there exists a constant $\kappa_0 = \kappa_0(n, \Lambda_0, p)$ such that for $\kappa \geq \kappa_0$ and $j \in \{0, \dots, n\}$, one has $\|\nabla (\Delta_j + \kappa)^{-1/2}\|_{p,p} < \infty$, with a norm depending only on $n, p, \Lambda_0$ and $\kappa$.
	\end{theorem}
	
	\begin{remark}
We can choose $\kappa_0 = C(n)(\Lambda_0+1)$, where $C(n)$ is a constant depending only on the dimension.
    A related result was obtained by Cheng, Thalmaier, and Wang in \cite{CTW23,CTW25} using a distinct approach. Their theorem holds for $1 < p \leq 2$, and for $p > 2$ under the condition that $|Rm|$ and $|\nabla Rm|$ belong to the Kato class.
	\end{remark}
	
	The Riesz transforms are fundamental tools for proving Calder\'{o}n-Zygmund inequalities. On a general Riemannian manifold, one observes that
    \begin{equation*}
        \begin{split}
            \mathrm{Hess}(\Delta_0 + \kappa)^{-1}=\nabla d(\Delta_0 + \kappa)^{-1/2}(\Delta_0 + \kappa)^{-1/2}
            =\nabla (\Delta_1 + \kappa)^{-1/2}d(\Delta_0 + \kappa)^{-1/2}.
        \end{split}
    \end{equation*}
    Hence,
	\begin{equation*}
		\|\mathrm{Hess}(\Delta_0 + \kappa)^{-1}\|_{p,p} \leq \|\nabla(\Delta_1 + \kappa)^{-1/2}\|_{p,p} \|d(\Delta_0 + \kappa)^{-1/2}\|_{p,p}.
	\end{equation*}
	If $-\Delta u=f$, then $(\Delta_0+\kappa) u=f+\kappa u$. By Theorem~\ref{thm1.6} 
    and Theorem~\ref{thm1.4},  we obtain
    \begin{equation*}
    \begin{split}
       \|\mathrm{Hess}u\|_{L^p}=\|\mathrm{Hess}(\Delta_0 + \kappa)^{-1}(f+\kappa u)\|_{L^p}
        \leq C\|(f+\kappa u)\|_{L^p}
       \leq C(\|f\|_{L^p}+\| u\|_{L^p}).
    \end{split}
    \end{equation*}
	In summary, we have the following theorem (see also \cite{HW26}).
	
	\begin{theorem}[Calder\'{o}n-Zygmund inequality]\label{thm1.7}
		Let $(M,g)$ be an $n$-dimensional Riemannian manifold such that $\|Rm\| \leq \Lambda_0$. Then for any $p \in (1, \infty)$, there exists a constant $C = C(n, p, \Lambda_0) > 0$ such that 
		\begin{equation*}\label{eq0.4}
			\|\mathrm{Hess}(u)\|_{L^p} \leq C(\|u\|_{L^p} + \|\Delta u\|_{L^p})
		\end{equation*}
		for any $u \in C^\infty_c(M)$.
	\end{theorem}
    
	\begin{remark}
		In \cite{GP15}, G\"uneysu and Pigola established Calderón–Zygmund inequalities on manifolds with Ricci curvature bounded from both sides and injectivity radius bounded below by a positive constant. Recently, Pigola \cite{Pig20} claimed that the Calder\'on-Zygmund inequalities hold for $p>\mathrm{max}(2,\frac{n}{2})$. Cao-Cheng-Thalmaier \cite{CCT21} showed that the inequality \eqref{eq0.4} holds if  $1<p\leq 2$ and $M$ has lower Ricci bound or $2<p<\infty$ and $\|Rm\|$ and $\|\nabla Rc\|$ are in the Kato class.
	\end{remark}
	Observing the operator factorization 
	\begin{equation*}
		\nabla (\Delta_0+\kappa)^{-1}\nabla^* = \left[\nabla(\Delta_0+\kappa)^{-\frac{1}{2}}\right] \left[(\Delta_0+\kappa)^{-\frac{1}{2}}\nabla^*\right],
	\end{equation*}
	the boundedness of the divergence-form Riesz transform follows from a duality argument. For any $h \in L^{p'}(M)$ with $1/p + 1/p' = 1$, we have
	\begin{equation*}
		\begin{split}
			\left| \langle (\Delta_0+\kappa)^{-\frac{1}{2}}\nabla^*f, h \rangle \right| &= \left| \langle f, \nabla (\Delta_0+\kappa)^{-\frac{1}{2}} h \rangle \right| \\
			&\leq \|f\|_{L^p} \|\nabla (\Delta_0+\kappa)^{-\frac{1}{2}} h \|_{L^{p'}} \\
			&\leq C \|f\|_{L^p} \|h\|_{L^{p'}},
		\end{split}
	\end{equation*}
	where the last inequality follows from the $L^{p'}$-boundedness of the covariant Riesz transform (see Theorem \ref{thm1.4}). Consequently, it follows that
	\begin{equation*}
		\| (\Delta_0+\kappa)^{-\frac{1}{2}}\nabla^*f\|_{L^p} \leq C\|f\|_{L^p}.
	\end{equation*}
	This estimate yields a streamlined proof of the Calder\'{o}n-Zygmund inequality in divergence form.
    
	\begin{theorem}\label{thm1.9}
		Let $(M,g)$ be an $n$-dimensional Riemannian manifold with $\|Rm\| \leq \Lambda_0$. Suppose $u$ satisfies $\Delta u = \nabla^* f$, where $\nabla^*$ is the formal adjoint of $\nabla$. Then for any $p \in (1, \infty)$, there exists a constant $C = C(n, p, \Lambda_0) > 0$ such that 
		\begin{equation*}
			\|\nabla u\|_{L^p} \leq C(\|u\|_{L^p} + \|f\|_{L^p})
		\end{equation*}
		for any $u \in C^\infty_c(M)$.
	\end{theorem}
	
	\noindent
\textbf{Acknowledgments:}

The authors are supported by Project of Stable Support for Youth Team in
Basic Research Field,  Chinese Academy of Sciences (YSBR-001), the National Natural Science Foundation of China (NSFC-12431003),  and a research fund from Hefei National Laboratory.

	\section{Preliminary}
	Let  $d$ be the geodesic distance on $M$, and $d\mu$ be the Riemannian measure. Denote by $B_r(x)$ the geodesic ball of center $x\in M$ and radius $r>0$ and by $|B_r(x)|$ or $V_x(r)$  its Riemannian volume $\mu(B_r(x))$. In this paper, we assume that $M$ is locally doubling, i.e.,
	for any $r>0$, there is a constant  $C(r)>0$ such that  the inequality
	\begin{equation*}
		|B_{2s}(p)|\leq C_r|B_{s}(p)|
	\end{equation*}
	holds for any  $p\in M$ and $0<s\leq r.$
	\subsection{Volume comparison theorem}
	\begin{theorem}\label{thm2.1}
		If $(M, g)$ is a complete Riemannian manifold with $Ric\geq -(n-1)\Lambda_0$,  and $q\in M$ is an arbitrary point. Then for any $0<r_1<r_2<+\infty$,
		\begin{equation*}
			\frac{|B_{r_2}(q)|}{|B_{r_1}(q)|}\leq \frac{|B^{\Lambda_0}_{r_2}|}{|B^{\Lambda_0}_{r_1}|}\leq Ce^{\sqrt{\Lambda_0}r_2}\left(\frac{r_2}{r_1}\right)^n.
		\end{equation*}
		where where $B_r^{\Lambda_0}$ is a geodesic ball of radius $r$ in the space form $M_{\Lambda_0}^n$ and $C$ depends only on $n,\Lambda_0$.
	\end{theorem}
	We also need the volume comparison theorem restricted to a ball.
	\begin{corollary}\label{coro2.2}
		If $(M, g)$ is a complete Riemannian manifold with $Ric\geq -(n-1)\Lambda_0$,  and $q\in M$ is an arbitrary point. Then there exists a constant $C>0$ such that for any $0<r\leq 4$ and $q_1\in B_{2}(q)$ with $d(p_1,p_2)\leq 10r$, we have 
		\begin{equation*}
			|B_{50r}(q_2)|\leq C|B_{5r}(q_1)\cap  B_2(q)|.
		\end{equation*}
	\end{corollary}
	\begin{proof}
		We first prove the corollary for $r<\frac{1}{5}$. There exists a point $q_3\in M$ such that $d(q,q_3)\leq d(q,q_1)-2r$ and $d(q_1,q_3)\leq 3r$. So, we have $B_r(q_3)\subset B_2(q)\cap B_{5r}(q_1)$ and $d(q_2,q_3)\leq 8r$. by Theorem \ref{thm2.1}, 
		\begin{equation*}
			\begin{split}
				|B_{50r}(q_2)| \leq |B_{r}(q_3)|\leq C|B_{r}(q_3)|\leq C|B_{5r}(q_1)\cap  B_2(q)|.
			\end{split}
		\end{equation*}
		For $\frac{1}{5}\leq r\leq 4$, there exists a point $q_3\in M$ such that $d(q,q_3)\leq d(q,q_1)-\frac{2}{5}$ and $d(q_1,q_3)\leq \frac{3}{5}$. So, we have $B_{\frac{1}{5}}(q_3)\subset B_2(q)\cap B_{5r}(q_1)$ and $d(q_2,q_3)\leq 8r$. Hence,
		\begin{equation*}
			\begin{split}
				|B_{50r}(q_2)| \leq |B_{53r}(q_3)|\leq C|B_{\frac{1}{5}}(q_3)|\leq C|B_{5r}(q_1)\cap  B_2(q)|.
			\end{split}
		\end{equation*}
	\end{proof}
	\subsection{Hardy–Littlewood maximal operators}
	Suppose that $(M,d,\mu)$ is a metric measure space. For notational
	convenience,  for any measurable subset $U$ of $M$ we often write $|U|$ instead of $\mu(U)$.
	we define  local version centered Hardy–Littlewood maximal functions on $M$.
      Fixing $q\in M$, we denote $B=B(q,1)$ and $2^iB=B(q,2^i)$ by $i\geq 1$.
	\begin{equation}
		\mathbf{M}(f)(x):=\sup_{0<r<4}\frac{1}{|B_r(x)|}\int_{B_r(x)}|f|d\mu=\sup_{0<r<4}\fint_{B_r(x)}|f|d\mu.
	\end{equation}
    \begin{remark}\label{rmk2.3}
        It is worth noting that if $\mathrm{supp}f\subset B_1(q)$ and $x\in B_2(q)$, then we have
        \begin{equation}
            \begin{split}
                &\sup_{4\leq r<\infty}\frac{1}{|B_r(x)|}\int_{B_r(x)}|f|d\mu\\
                =&\sup_{4\leq r<\infty}\frac{1}{|B_r(x)|}\int_{B_1(q)}|f|d\mu
                \leq \frac{1}{|B_4(x)|}\int_{B_1(q)}|f|d\mu
                \leq \mathbf{M}(f)(x).
            \end{split}
        \end{equation}
    \end{remark}

	If $M$ is locally doubling, then $\mathbf{M}$ is of weak type $(1,1)$ and strong $(p,p)$ for $p\in (1,+\infty]$. In particular, if $M$ has Ricci curvature bounded from below,
	then, by the relative volume estimate, the Riemannian
	measure is locally doubling, $\mathbf{M}$ is of weak type $(1,1)$ and strong $(p,p)$ for $p\in (1,+\infty]$.
	Then we get the following Hardy–Littlewood maximal inequality:
	\begin{theorem}\cite{Ste70,Med25}\label{thm2.3}
		Suppose that $M$ is a Riemannian manifold and $M$ such that $Ric\geq -(n-1)\Lambda_0$. Then  
		\begin{equation*}
			\begin{split}
				\|\mathbf{M}(v)\|_{L^p}\leq C\|v\|_{L^p}\quad \text{for any }1<p\leq +\infty,
			\end{split}
		\end{equation*}
		and 
		\begin{equation*}
			\begin{split}
				|\{y\in M:\mathbf{M}v(y)\geq \lambda\}|\leq \frac{C}{\lambda}\|v\|_{L^1}.
			\end{split}
		\end{equation*}
	\end{theorem} 
	\subsection{Hodge Laplacian and rough Laplacian }
	The Bochner-Lichnerowicz formula for Hodge Laplacian writes
	\begin{equation}\label{eq2-10}
		\Delta_j=\nabla^*\nabla +V_j
	\end{equation}
	where 
	\begin{equation*}
		V_j\in \Gamma_{C^\infty}(M,\mathrm{End}(\Lambda^jT^*M))
	\end{equation*}
	is a fiberwise self-adjoint 0-th order operator, which satisfies
	\begin{equation*}
		|V_j|\leq C|\mathrm{Riem}|,
	\end{equation*}
	where $C=C(n)>0$ is a constant that depends on $n$.
	\subsection{Harmonic radius}
	
	\begin{definition}($W^{2k,k}_p$ harmonic radius)\label{def2.4} The $W^{2k,k}_p$ harmonic radius at $x$ and $n+2<p<\infty$, is the supremum of all $R>0$ such that  there exists a coordinate chart $\phi:B_R(x)\to \IR^n$ satisfying 
		\begin{itemize}
			\item $\Delta \phi^j=0$ on $B_{R}(x)$  and $\phi^j(x)=0$ for each $j$,
			\item $2^{-1}(\delta_{ij})\leq (g_{ij})\leq 2 (\delta_{ij})$,  in $B_{R}(x)$ as symmetric bilinear forms,
			\item $\sum_{1\leq |J|\leq k}R^{|J|-\frac{n}{p}}\|\partial^J g_{ij}\|_{L^p(B_R(x))}\leq 1$.
		\end{itemize}
		We denote this radius by $r_{k,p}(x)$.
	\end{definition}
	\begin{proposition}\label{prop1.8}\cite{And90}
		Let $B_{2r}(x)$ be a compact ball in a Riemannian manifold $(M,g)$. Suppose that there are numbers 
        $k\in \IN_{> 0},\varepsilon,r>0,c_0,\cdots,c_k>0$ with 
		\begin{equation*}
			|\nabla^j\mathrm{Rm}(y)|\leq c_j,r_{inj}(y)\geq r,  \; 
            \forall \; y\in B_{2r}(x), j\in\{0,\cdots,k-1\}.
		\end{equation*}
		Then there exists a positive constant $c=c(n,k,\alpha,r,c_1,\cdots,c_k)$ such that 
		\begin{equation*}
			r_{k,p}(x)\geq c. 
		\end{equation*}
	\end{proposition}
	
\begin{definition}\label{def4.1}
Suppose $0<\alpha<1$ and $l$ is a nonnegative integer. Suppose $s$ is a smooth section of a vector bundle $\mathcal{E}\to M$. Define
		\begin{equation*}
			[s]^{(r)}_{l+\alpha}(x) :=\sup_{d(x,x')<\min  \{r,4\}}\sup_{\gamma\in \Xi(x,x')}\frac{|\nabla^ls(x)-\tau_\gamma(x,x')\nabla^ls(x')|}{d^\alpha(x,x')}<\infty,
		\end{equation*}
        where $$\Xi(x,x') :=\{\text{the shortest geodesic line connecting }x \text{ and }x'\}.$$
		The operator $\tau_\gamma(x, x'): T_{x'}M \to T_xM$ represents the parallel displacement of the tensor field along the path $\gamma$, mapping the tangent space at $x'$ to the tangent space at $x$.
	\end{definition}

	\section{Heat equations in $W^{2,1}_p$ atalas}
	\label{sec3}
	
	In this section, we study fundamental estimates of heat solutions in harmonic coordinates.

	By comparing the geodesic curve connecting a boundary point to the origin, it is not hard to see from the second condition of Definition~\ref{def2.4} that 
	\begin{align}
		B_{0.5R}(0) \subset  B_{\frac{\sqrt{2}}{2}R}(0) \subset \phi_R (B_R(x)) \subset B_{\sqrt{2}R}(0)
		\label{eq3.1}   
	\end{align}
	for every $R<r_{k,p}(x)$.  
	Therefore, the Euclidean ball $B_{0.5R}(0) \subset \mathbb{R}^n$ is equipped with two metrics, the Euclidean metric $g_E$ and the push-forward metric $\phi_R^* g$.   For simplicity of notation, we still denote $\phi_R^* g$ by $g$.  In other words, we can regard the identity map from $(B_{0.5R}(0), g)$ to $(B_{0.5R}(0), g_E)$ as a harmonic map.   
	Thus, we have
	\begin{align}
		0=\Delta_g x^k=g^{ij} \left( \frac{\partial^2 x^k}{\partial x^i \partial x^j} -\Gamma_{ij}^l \frac{\partial x^k}{\partial x^l} \right)=-g^{ij}\Gamma_{ij}^k. 
		\label{eq3.2}    
	\end{align}
	The tensor $g_{ij}$ is now a matrix-valued function. In light of the relationship (\ref{eq3.1}), it follows from the definition of harmonic radius that the following inequalities hold. 
	\begin{align}
		&\frac12 \sum_{i=1}^n (V^i)^2 \leq   g_{ij}V^i V^j \leq 2\sum_{i=1}^n (V^i)^2, \quad \forall \; V \in \IR^n; \label{eq3.3} \\
		&\sum_{1\leq |J|\leq k}R^{|J|-\frac{1}{4}}\|\partial^J g_{ij}\|_{L^p(\ID)}\leq 1. \label{eq3.4} 
	\end{align}
	
	Suppose $R \geq 100$. Then $B_{0.5R}(0) \supset B_2(0)$. 
	Define
	\begin{align}
		\begin{cases}
			&B:= B_2(0), \quad B' :=B_1(0); \\
			&\Omega := B \times [0, 4], \quad \Omega' :=B' \times [3, 4]. 
		\end{cases}
		\label{eq3.5}    
	\end{align}
	
	\begin{definition}
		A local model space-time is a smooth family of metrics $g$ defined on $\Omega$ satisfying 
		\begin{align*}
			&\Delta_{g(0)} x^i=0 \; \textrm{on} \; B, \quad \forall \; i \in \{1, 2, \cdots, n\};\\
			&\sup_{\Omega} |Rm|(x) \leq \xi^2. 
		\end{align*}
		Here 
		\begin{align}
			0<\xi=\xi(n)<\frac{1}{100n \pi}     \label{eq3.8}
		\end{align}
		is a small positive constant such that the $W^{2,1}_p$ harmonic radius of $g$ is at least $100$.
		
		Furthermore,  we require that $g|_{B}$ satisfies (\ref{eq3.3}) and (\ref{eq3.4}) with $k=2$. 
		\label{def3.1}
	\end{definition}
	
	We shall study the behavior of the heat solutions in local model space-time. 
	
	The following definitions of the parabolic Sobolev and H\"older norms are well known.
	
	\begin{align}
		&\|u\|_{W_p^{2m,m}(\Omega)} 
		:=\sum_{|J|+2k\leq  2m} \|D^J \partial_t^ku\|_{L^p}, \label{eq3.9}\\
		&\|u\|_{C^{m+\alpha, \frac{m+\alpha}{2}}(\Omega)}
		:=\sum_{|J|+2k\leq m} \left(\|D^J\partial_t^ku\|_{C^0}+r^{\alpha}[D^J \partial_t^ku]_{\alpha,\frac{\alpha}{2}} \right),   \label{eq3.10}   
	\end{align}
	where $J$ is a multi-index.  We have the following Sobolev embedding properties. 
	
	\begin{lemma}
		\cite[Theorem 10.4.10]{Kry24}
		\label{lm3.2}
		For any $m \in \IN$, $n+2<p<\infty$ and $\alpha=1-(n+2)/p$, we have 
		\begin{align}
			W_p^{2m,m}(\Omega) \hookrightarrow C^{2m-1+\alpha,\frac{2m-1+\alpha}{2}}(\Omega).
			\label{eq3.11}     
		\end{align}
		In particular, there is a constant $C=C(n,m,p)>0$ such that 
		\begin{align}
			\|u\|_{C^{2m-1+\alpha,\frac{2m-1+\alpha}{2}}}
			\leq C  \|u\|_{W_p^{2m,m}(\Omega)}.
			\label{eq3.12}     
		\end{align}
	\end{lemma}

	\begin{lemma}\cite[Theorem 7.2]{Sch}
		\label{lm3.3}
		Let $u\in C^{2,1}(\Omega)$ be a classical solution of the following parabolic system
		\begin{align*}
			\begin{cases}
				&u_t^{\mu}-A_{ij}^{\mu\nu}D_{ij}u^\nu +B_i^{\mu\nu}D_iu^{\nu}+C^{\mu\nu}u^\nu=f^\mu, \quad \text{in} \; \Omega;  \\
				&u=0, \quad \text{on} \; \partial \Omega.    
			\end{cases}
		\end{align*}
            Suppose that the following conditions hold
		\begin{align*}
			\begin{cases}
				&A_{ij}^{\mu\nu}\in C^0(\Omega); \\
				&\Theta^{-1}|V|^2\leq a^{\mu\nu}_{ij}V^i_\mu V^j_\nu\leq \Theta |V|^2,\quad \forall \; V \in \mathbb{R}^n;  \\
				&\|A_{ij}^{\mu\nu}\|_{L^\infty(\Omega)}+\|B_i^{\mu\nu}\|_{L^\infty(\Omega
					)}+\|C^{\mu\nu}\|_{L^\infty(\Omega)} \leq \Theta.
			\end{cases}    
		\end{align*}
		Then
		\begin{equation*}
			\begin{split}
				\|u\|_{W_p^{2,1}(\Omega)}
				\leq C \left\{\|u\|_{L^p(\Omega)}+\|f\|_{L^p(\Omega)} \right\}, 
			\end{split}
		\end{equation*}
		where $C=C(n,p,\Theta)$. 
	\end{lemma}
	
	Then we start to analyze the regularity of the metrics and PDE solutions on $\Omega$.    
	
	\begin{lemma}
		\label{lm3.4}
		Based on the choice of $\Omega$ and $g$, the following estimates hold.
		\begin{align}
			\|g\|_{C^{1+\alpha, \frac{1+\alpha}{2}}(\Omega)} \leq C. 
			\label{eq3.14}    
		\end{align}
	\end{lemma}

	\begin{lemma}\label{lm3.5}
		Suppose $u$ is a smooth tensor on $B$.  Then 
		\begin{align}
			\sup_{x \in B'} \left\{ |\nabla u|_g(x) +  +[ \nabla u]_{g,\alpha}(x) \right\} \leq C(n) \|u\|_{C^{1+\alpha}(B)}. 
			\label{eq3.15}  
		\end{align}
		Note that 
		\begin{align}
			[\nabla u]_{g,\alpha}(x) := \sup_{y \in B_{\frac12}(x)}  \frac{|\nabla u(x)-\tau_{\gamma}(x,y)\nabla u(y)|}{|\gamma|^{\alpha}},   
			\label{eq3.16}  
		\end{align}
		where $\gamma$ is any shortest geodesic connecting $x$ and $y$ under the metric $g$, $\tau_{\gamma}(x,y)$ is the parallel transportation from $y$ to $x$ along the geodesic $\gamma$. 
	\end{lemma}
	
	\begin{proof}
		On $B'$, it is clear that
		\begin{align}
			\frac12 |\partial u|^2  
			\leq  |\nabla u|_g^2= g^{ij}u_i u_j \leq 2 |\partial u|^2 
			\label{eq3.17}  
		\end{align}
		holds point-wisely. 
	
		Now fix $x\in B'$ and $y \in B_{\frac12}(x)$.  
		Let $\gamma$ be a unit-speed shortest geodesic connecting $x$ and $y$ such that $\gamma(0)=x$ and $\gamma(L)=y$.  Under the small curvature assumption, we know $\gamma \subset B$. 
		We want to show
		\begin{align}
			|u_{i_1\cdots i_j,i}(x)-\tau_{\gamma}(x,y)u_{i_1\cdots i_j,i}(y)| \leq C \|u\|_{C^{1+\alpha}(B)} \cdot L^{\alpha}. 
			\label{eq3.18}   
		\end{align}
		For simplicity of notation, we define
		\begin{equation*}
			\Upsilon(\gamma(\tau)):=\tau_\gamma(\gamma(\tau),y) \nabla u(y)=\Upsilon_{i_1\cdots i_{j+1}}(\gamma(\tau))dx^{i_1}\otimes dx^{i_2} \cdots\otimes  dx^{i_j}\otimes  dx^{j+1}.
		\end{equation*}
		From definition, it is clear that $\Upsilon_{i_1\cdots i_j,i}(y)=u_{i_1\cdots i_j,i}(y)$. 
		Then we have 
		\begin{equation}
			\begin{split}
				&\quad |u_{i_1\cdots i_j,i}(x)-\tau_{\gamma}(x,y)u_{i_1\cdots i_j,i}(y)|\\
				&=|u_{i_1\cdots i_j,i}(y)(x)-\Upsilon_{i_1\cdots i_j,i}(x)|\\
				&\leq |u_{i_1\cdots i_j,i}(y)(x)-\Upsilon_{i_1\cdots i_j,i}(y)| + |\Upsilon_{i_1\cdots i_j,i}(y)-\Upsilon_{i_1\cdots i_j,i}(x)|\\
				&=\underbrace{|u_{i_1\cdots i_j,i}(y)(x)-u_{i_1\cdots i_j,i}(y)(y)|}_{I} + \underbrace{|\Upsilon_{i_1\cdots i_j,i}(y)-\Upsilon_{i_1\cdots i_j,i}(x)|}_{II}. 
			\end{split}
			\label{eq3.20}
		\end{equation}
		For part $I$, we know 
		\begin{equation*}
			\begin{split}
				I&=\left|\partial_i u_{i_1\cdots i_j}(x)-\partial_i u_{i_1\cdots i_j}(y) - (\Gamma_{ii_k}^m(x) u_{i_1\cdots m\cdots i_j}(x) -\Gamma_{ii_k}^m(y)u_{i_1\cdots m\cdots i_j}(y)) \right|\\
				&\leq C L^{\alpha} \left\{ \|\partial u\|_{C^{\alpha}(B)} + \|\partial g\|_{C^0(B)}
				\| u\|_{C^{\alpha}(B)} + \| u\|_{C^0(B)}
				\|\partial g\|_{C^{\alpha}(B)} \right\}. 
			\end{split}    
		\end{equation*}
		In short, we have
		\begin{align}
			I \leq CL^{\alpha} \|u\|_{C^{1+\alpha}(B)}. 
			\label{eq3.21}
		\end{align}
		For part II, we note that $\Upsilon$ satisfies the following equation of parallel transportation
		\begin{equation}\label{eq3.22}
			\begin{split}
				\frac{d \Upsilon_{i_1 \dots i_{j+1}}}{d\tau}-\sum_{r=1}^j \Gamma_{ki_r}^m \Upsilon_{i_1 \dots (m)_r \dots i_{j+1}} \frac{dx^k(\gamma(\tau))}{d\tau}=0
			\end{split}
		\end{equation}
		Note that 
		\begin{align*}
			\left\|\Gamma^{j}_{l k}\frac{dx^k(\gamma(\tau))}{d\tau} \right\|_{C^0(\gamma)} 
			\leq C \|\partial g\|_{C^0(B)}
			\leq C. 
		\end{align*}
		Applying Gronwall's inequality (cf. (\ref{eq3.26}) in Lemma~\ref{lm3.6}) on the ODE system (\ref{eq3.22}), we obtain
		\begin{align*}
			\|\Upsilon\|_{C^0(\gamma)} \leq C |\Upsilon(y)|  
			\leq C \|u\|_{C^0(B)} \leq C \|u\|_{C^{1+\alpha}(B)}. 
		\end{align*}
		Thus, the standard ODE estimates (cf. (\ref{eq3.27}) in Lemma~\ref{lm3.6}) imply that 
		\begin{align}
			II \leq C \|\Upsilon\|_{C^0(\gamma)} \cdot L
			\leq C \|u\|_{C^{2+\alpha}(B)} \cdot L^{\alpha}, 
			\label{eq3.24}
		\end{align}
		where we used the facts $L \in (0, 2)$ and $\alpha \in (0,1)$. 
		
		Plugging (\ref{eq3.21}) and (\ref{eq3.24}) into (\ref{eq3.20}), we obtain (\ref{eq3.18}). 
		Combining (\ref{eq3.18}) with (\ref{eq3.17}) and (\ref{eq3.18}), we arrive at (\ref{eq3.15}). 
	\end{proof}

	\begin{lemma}
		\label{lm3.6}
		Let $A\in C([0,L],\IR^{n^j\times n^j})$ and $x(\tau) \in C^1([0,L],\IR^{n^j})$ satisfy
		\begin{align}
			\dot{x}(\tau)=A(\tau)x(\tau),\quad x(0)=x_0.
			\label{eq3.25}    
		\end{align}
		Then the following estimate holds:
		\begin{align}
			&|x(\tau)| \leq \exp\bigg(\int_{0}^{\tau}|A(s)|ds\bigg) \cdot |x_0|, \label{eq3.26} \\
			&|x(\tau)-x(0)| \leq  \|A\|_{C^0[0,L]} \cdot \|x\|_{C^0[0,L]} \cdot \tau. \label{eq3.27}
		\end{align}
	\end{lemma}
	
	\begin{proof}
		The differential equation (\ref{eq3.25}) can be rewritten as the integral equation
		\begin{align}
			x(\tau) =x_0+\int_{0}^{\tau}A(s)x(s)ds, 
			\label{eq3.28}    
		\end{align}
		which implies
		\begin{align*}
			|x(\tau)|\leq |x_0|+\int_{0}^{\tau} |A(s)|\cdot |x(s)| ds.
		\end{align*}
		Applying Gronwall's inequality to the above inequality, we obtain (\ref{eq3.26}).
		It is clear that (\ref{eq3.27}) follows from (\ref{eq3.28}). 
	\end{proof}

	\begin{proposition}\label{prop3.7}
		Suppose $u=u_{i_1\cdots i_j}dx^{i_1}\otimes dx^{i_2} \cdots\otimes  dx^{i_j}$ satisfies
		\begin{align}
			(\partial_t -\Delta_j) u=0, \quad \textrm{on} \quad \Omega. 
			\label{eq3.29}    
		\end{align}
		Then 
		\begin{align}
			\|u\|_{W_p^{2,1}(\Omega')} \leq C(n)\|u\|_{L^{\infty}(\Omega)}. 
			\label{eq3.30}
		\end{align}    
	\end{proposition}
	
	\begin{proof}
		In the given $W_p^{2,1}$-harmonic coordinate system of $g$,  equation (\ref{eq3.29}) can be written as 
			\begin{equation*}
			\begin{split}
				\partial_t u_{i_1\cdots  i_j}&-g^{kl}\partial^2_{kl}u_{i_1\cdots  i_j}-g^{kl}\sum_{s=1}^j\partial_{k}u_{i_1\cdots i_{s-1}mi_{s+1}\cdots i_j}\Gamma^m_{i_sl}\\
				&-g^{kl}\sum_{s=1}^j\partial_{l}u_{i_1\cdots i_{s-1}mi_{s+1}\cdots i_j}\Gamma^m_{i_sk}+g^{kl}\partial_mu_{i_1\cdots  i_j}\Gamma^m_{kl}  \\
				&=g^{kl}\sum_{s=1}^ju_{i_1\cdots i_{s-1}ni_{s+1}\cdots i_j}\Gamma^n_{i_sm}\Gamma^m_{kl}     \\  
				&-g^{kl}\sum_{s=1}^j(\partial_l\Gamma^m_{i_sk}-\Gamma^n_{i_sl}\Gamma^m_{nk}) u_{i_1\cdots i_{s-1}mi_{s+1}\cdots i_j}+(Rm*u)_{i_1\cdots  i_j}, 
			\end{split}
		\end{equation*}
		Note that $\frac12 \delta_{kl} \leq g_{kl} \leq 2 \delta_{kl}$ in $B$, which guaranties the uniform parabolic condition. 
		Also note that $\|\Gamma\|_{L^{\infty}(\Omega)} \leq C(n)$ by (\ref{eq3.14}). 
		Furthermore, the right hand side can be bounded by
		\begin{align*}
			 C\|u\|_{L^{\infty}(\Omega)} \cdot \left\{1+\|\partial \partial g\|_{L^p(\Omega)}\right\}
			\leq C \|u\|_{L^{\infty}(\Omega)}. 
		\end{align*}
		Therefore, by the $W_p^{2,1}$ estimate for parabolic systems (cf. Lemma~\ref{lm3.3}), we have
		\begin{align*}
			\|u\|_{W_p^{2,1}(\Omega')} 
			\leq C  \|u\|_{L^{\infty}(\Omega)},  
		\end{align*}
		which is exactly (\ref{eq3.30}). 
	\end{proof}

	\begin{theorem}\label{thm3.8}
	Suppose $u=u_{i_1\cdots i_j}dx^{i_1}\otimes dx^{i_2} \cdots\otimes  dx^{i_j}$ satisfies
	\begin{align*}
		(\partial_t-\Delta_j)u=0.
	\end{align*}
		Then we have 
		\begin{align}
			\sup_{(x,t) \in \Omega'} \left\{ |\nabla u|_g +[\nabla u]_{g,\alpha}
			\right\} \leq
			C(n)  \|u\|_{L^{\infty}(\Omega)}. 
			\label{eq3.33}    
		\end{align} 
	\end{theorem}
	
	\begin{proof}
	By Proposition \ref{prop3.7},
		\begin{equation*}
			\begin{split}
					 \|u\|_{W_p^{2,1}(\Omega)} \leq C \|u\|_{L^{\infty}(\Omega)}. 
			\end{split}
		\end{equation*}
		Then the parabolic Sobolev embedding implies that
		\begin{align*}
			\sup_{s \in [-1,0]} \|\partial u(\cdot, s)\|_{C^{\alpha}(B')} 
			\leq C(n)  \|u\|_{L^{\infty}(\Omega)}.   
		\end{align*}
		Then by a similar estimate as in the proof of Lemma~\ref{lm3.5}, we obtain
		\begin{align}
			\sup_{s \in [-1,0]}\bigg\{ \|\nabla u(\cdot, s)\|_{L^\infty(B')} +\|\nabla u(\cdot, s)\|_{C_g^{\alpha}(B')} \bigg\}\leq C(n)  \|u\|_{L^{\infty}(\Omega)}.    
			\label{eq3.37}        
		\end{align}
		Therefore, (\ref{eq3.33}) follows from  (\ref{eq3.37}).    
	\end{proof}

	\section{Heat kernel estimate}
	\label{sec4}

	Note that the curvature condition itself cannot exclude the happening of collapsing. 
	In other words, the cut radius could be very small.   This phenomenon causes the major difficulty of this section. However, it can be overcome by standard technique: we shall lift the metrics $g$ locally to some tangent space $T_x M$.

	\begin{lemma}
		\label{lm4.2}
		Suppose $M$ is an evolving manifold satisfying $\|Rm\| \leq \Lambda_0$. 
		Let $H_j(x,y,t)$ be the heat kernel of $\Delta_j$. Then there is a constant $C_1=C_1(n,\Lambda_0)>0$ and $C_2=C_2(n)>0$
		\begin{equation}\label{eq4.2}
			\begin{split}
				H_j(x,y,t)\leq C_1e^{C_2\Lambda_0t}V_{x}(\sqrt{t})^{-1}e^{-\frac{d^2(x,y)}{C_2t}}.
			\end{split}
		\end{equation}
	\end{lemma}
	
	\begin{proof}
 Noting that  
    \begin{equation*}
       \Delta_j=\nabla^*\nabla +V_j
    \end{equation*}
    where $|V_j|\leq C(n)\|Rm\|$.  By Kato's inequality,
    \begin{equation*}
        (\partial_t-\Delta) |H_j| \leq C(n)\|Rm\||H_j|.
    \end{equation*}
	Then~\eqref{eq4.2} follows directly from the volume comparison and Li-Yau estimate. 
	\end{proof}
	
	Note that (\ref{eq4.2}) can be rewritten as 
	\begin{align*}
		t^{\frac{n}{2}} H_j(x,y,t) \leq C  \left\{\frac{V_{x}(\sqrt{t})}{(\sqrt{t})^n}\right\}^{-1} e^{-\frac{d^2(x,y)}{C_2t}}, 
	\end{align*}
	which has the advantage that both sides are scaling invariant.  Without loss of generality, for any $0<t<1$, we can rescale the manifold by $\lambda=4\Lambda_0 \xi^{-2}t^{-1}$.  Thus, we obtain a manifold $(M,\tilde{g})$ that satisfies
	\begin{align*}
		\sup_{M} |Rm|_{\tilde{g}} \leq  \xi^2.  
	\end{align*}
	Using the rescaling property and volume comparison, the heat kernel $\tilde{H}_j$ of $(M,\tilde{g})$ satisfies the estimate
	\begin{align*}
		\tilde{H}_j(x,y,4) \leq C V_{x}(1)^{-1} e^{-\frac{d^2(x,y)}{4C_2}}. 
	\end{align*}
	We claim that similar estimates hold in a neighborhood of $(x, 4)$: 
	\begin{align}
		\sup_{w \in B_{10}^{\tilde{g}}(x), 3 \leq t \leq 4} 
		\tilde{H}_j(w,y,t) \leq C V_{x}(1)^{-1} e^{-\frac{d^2(x,y)}{8C_2}}  =:\Theta.
		\label{eq4.3}
	\end{align}
	In fact, the triangle inequality implies that
	\begin{align*}
		d^2(w,y)\geq \frac{1}{2}d^2(x,y)-d^2(w,x)\geq \frac12 d^2(x,y)-100.
	\end{align*}
	Adjusting $C$, we have
	\begin{align*}
		e^{-\frac{d^2(w,y)}{C_2s}}\leq e^{-\frac{\frac12 d^2(x,y)-100}{C_2s}}\leq Ce^{-\frac{d^2(x,y)}{8C_2}}, 
		\quad \forall \;   w \in B_{10}^{\tilde{g}}(x), 3 \leq s \leq 4. 
	\end{align*}
	Thus, we finished the proof of (\ref{eq4.3}). 
	
	Fix $x$ and define $\varphi$ as the exponential map from $T_x M$ to $M$, with respect to the metric $\tilde{g}$:
	\begin{align*}
		\varphi(v) := Exp_{x}^{\tilde{g}}(v), \quad \; \forall\; v \in T_x M \simeq \mathbb{R}^n. 
	\end{align*}
	Using this smooth map $\varphi$, we can pull back the evolving metrics $\tilde{g}$ to $T_x M$ as
	\begin{align*}
		\hat{g} :=\varphi^* \tilde{g}, 
	\end{align*}
	Since $\xi$ is chosen sufficiently small,  $\log |d \varphi|$ is uniformly bounded. 
	Consider $(B_{r}(0), \hat{g})$ with $r=\frac{1}{100n\pi \xi}>>100$. 
	We can further assume the $W_p^{2,1}$-harmonic radius at $0$ is at least $100$.  
	Thus,  we have the harmonic diffeomorphism $\psi: B_{10}(0) \to \mathbb{R}^n$.  
	Then 
	\begin{align*}
		\psi: B_{10}^{\varphi^* \tilde{g}}(x) \mapsto \mathbb{R}^n. 
	\end{align*}
	According to the definition of harmonic radius, we have 
	\begin{align}\label{eq4.4}
		B_{5\sqrt{2}}(0) \subset  \psi (B_{10}^{\varphi^* \tilde{g}}(x)) \subset B_{10\sqrt{2}}(0). 
	\end{align}
  Define $\Omega=B_1(0)\times [3,4]$, we have
	\begin{align*}
		(\Omega, \bar{g}) \; \textit{is a model space-time in the sense of Definition~\ref{def3.1}}.   
	\end{align*}

	Fix $y$ and define
	\begin{align*}
		\bar{H}_j(z,t) :=\psi_* \varphi^* \tilde{H}_j(z,y,t)=\tilde{H}_j(\varphi \psi^{-1}(z),y,t).
	\end{align*}
	Then $\bar{H}_j$ is a heat solution on $(\Omega, \bar{g})$.  It follows from (\ref{eq4.2}) and (\ref{eq4.4}) that 
	\begin{align}
		\sup_{\Omega} \bar{H}_j \leq \Theta. 
		\label{eqn:HA10_5}    
	\end{align}
	Thus, we can apply Proposition~\ref{prop3.7} and Lemma \ref{lm3.2} to obtain
	\begin{align}
		\|\bar{H}\|_{C^{1+\alpha,\frac{1+\alpha}{2}}(\Omega')} \leq C(n)\Theta, 
		\label{eqn:HA10_6}    
	\end{align}
where $\Omega'=B_{\frac12}(0)\times [4-\frac14,4]$,	which then implies that (cf. Lemma \ref{lm3.5})
	\begin{align}
		\sup_{(x,t) \in \Omega'} \left\{ |\nabla \bar{H}_j|_{\bar{g}} 
		+[ \nabla \bar{H}_j]_{\bar{g},\alpha}
		\right\} \leq
		C(n) \Theta. 
		\label{eq4.8}    
	\end{align}
	In particular, we have 
	\begin{align}
		\left. \left\{ |\nabla \bar{H}_j|_{\bar{g}} + 
		+[\nabla \bar{H}_j]_{\bar{g},\alpha}
		\right\} \right|_{(0,4)}\leq
		C(n) \Theta. 
		\label{eq4.9}    
	\end{align}
    
 For $t > 1$, we choose $\lambda = 4\Lambda_0 \xi^{-2}$  and repeat the argument above.
 
	Recalling the definition of $\Theta$ in (\ref{eq4.3}) and rescaling back to the original $g$, we obtain the following Theorem.

	\begin{theorem}\label{thm4-3}
		Let $M$ be a complete Riemannian manifold such that $\|Rm\|<\Lambda_0$. Let $H_j$ be the heat kernel of $\Delta_j$ of $M$. 
		Then there are $\alpha \in (0,1)$ and $C_1=C_1(n,\Lambda_0,\alpha)>0,C_2=C_2(n)>0$ such that 
		\begin{itemize}
			\item [(i)] $C^1$-estimate
			\begin{equation}\label{eqA11}
				\begin{split}
					|H_j(x,y,t)|+t^{\frac12}|\nabla_x H_j(x,y,t)|\leq C_1V^{-1}_x(\sqrt{t})e^{C_2\Lambda_0t}e^{-\frac{d^2(x,y)}{C_2t}},
				\end{split}
			\end{equation}
            for any $x,y\in M$ and $0<t<\infty$.
			\item [(ii)]
			There exists a constant $\rho=\rho(n,\Lambda_0)>0$ such that
			\begin{equation}\label{eqA12}
				\begin{split}
					t^{\frac{1+\alpha}{2}}[\nabla_x H_j(x,y,t)]^{\rho \sqrt{t}}_\alpha\leq C_1V^{-1}_x(\sqrt{t})e^{C_2\Lambda_0 t}e^{-\frac{d^2(x,y)}{C_2t}},
				\end{split}
			\end{equation}
		\end{itemize}
		 for any $x,y\in M $.
	\end{theorem}

	\section{Boundedness of Riesz type operator}
	In this section, we get the $L^p$ boundedness of Riesz type operator for $p>2$. 

    We fix a point $q \in M$ and recall that $B=B(q,1)$ and $2^iB=B(q,2^i)$ for $i\geq 1$.
    
	\begin{definition}\label{def5.1}
		Let $(M,g)$ be an n-dimensional Riemannian manifold.	
        A section $\mathcal{K}:M\times M\times \IR^+\to   \otimes^k T^*M\times \otimes^lT^*M $ is said to be a Riesz type 
		kernel if it satisfies the following conditions for some constants $C_1,C_2>0$.
		\begin{itemize}
			\item [(a)] $|\mathcal{K}(x,y,t)|\leq C_2V^{-1}_x(\sqrt{t})t^{-\frac{1}{2}}e^{-C_0C_1t}e^{-\frac{d^2(x,y)}{C_1t}}$.
			\item [(b)] There exists a constant $\rho>0$ and $\alpha\in (0,1)$ such that
			\begin{equation}\label{eq5.1}
				\begin{split}
					[\nabla_x \mathcal{K}(x,y,t)]^{\rho \sqrt{t}}_\alpha\leq C_2t^{-\frac{1+\alpha}{2}}V^{-1}_x(\sqrt{t})e^{-C_0C_1t}e^{-\frac{d^2(x,y)}{C_1t}},
				\end{split}
			\end{equation}
		 for any $x,y\in M $.
		\end{itemize}
		Here $C_0=64n(\sqrt{\Lambda_0}+1)^2$.
		$T$ is said to be a singular integral operator associated with the Riesz type kernel $K$ if
		\begin{equation*}
			\begin{split}
				Tf(x)=\int_0^\infty t^{-\frac12}\int_M \mathcal{K}(x,y,t)f(y)dydt
			\end{split}
		\end{equation*}
		whenever $f\in\Gamma^\infty_c(M,\otimes^lT^*M)$.
		
		A singular integral $T$ associated to a Riesz type kernel $K$ is called a Riesz type operator when it is bounded on $L^2$, that is, there is a $C_2>0$ such that
		\begin{equation}\label{eq5.3}
			\begin{split}
				\|Tf\|_{L^2}\leq C_3\|f\|_{L^2}
			\end{split}
		\end{equation}
		for all smooth compactly supported $f$.
	\end{definition}
		For a Riesz type operator $T$, we have
	\begin{theorem}\label{thm5.2}
		For $2< p<\infty$, there exists a constant  $C=C(n,p,\rho,\alpha,\Lambda_0,C_1,C_2,C_3)>0$ such that 
		\begin{equation*}
			\begin{split}
				&\int_{M}|Tf|^p \leq C\int_{B}|f|^p
			\end{split}
		\end{equation*}
		for any $f\in L^p(\otimes^lT^*M)$ with $\mathrm{supp}f\subset B$.
	\end{theorem}
	\begin{proof}
		The theorem follows from Proposition \ref{prop5.3} and Proposition \ref{lm5.9} below.
	\end{proof}

	\begin{proposition}
        \label{prop5.3}
		For $2< p<\infty$, there exists a constant  $C=C(n,p,\rho,\alpha,\Lambda_0,C_1,C_2,C_3)>0$ such that 
		\begin{equation}
        \label{eqn:HC21_5}
			\begin{split}
				&\int_{2B}|Tf|^p \leq C\int_{B}|f|^p.
			\end{split}
		\end{equation}
		for any $f\in L^p(M,\otimes^jT^*M)$ with $\mathrm{supp}f\subset B$.
	\end{proposition}
	Before proving the proposition, we need the following lemmas.
	\begin{lemma}
    \label{lm5.4}
		Let $M$ be a Riemannian manifold with $Ric \geq -(n-1)\Lambda_0$. 
        Suppose $y_0\in 2B$. Suppose $f\in L^2(B)$ and $\mathrm{supp}f\subset B\setminus B_{5r}(y_0)$ for some $r \in (0, 4)$.
        Then there exists a positive constant $C=C(n,\rho,\alpha,\Lambda_0,C_0,C_1,C_2)$ such that
		\begin{equation}
        \label{eqn:HC17_4}
			\begin{split}
			||Tf|(y)-|Tf|(y_0)|\leq C\mathbf{M}(|f|)(y_0)
			\end{split}
		\end{equation}
          for all $y \in B_{4r}(y_0)\cap 2B$. 
	\end{lemma}
	
	\begin{figure}[htbp]
		\centering
		\resizebox{0.3\textwidth}{!}{
			\tikzset{every picture/.style={line width=0.75pt}} 

\begin{tikzpicture}[x=0.75pt,y=0.75pt,yscale=-1,xscale=1]

\draw  [fill={rgb, 255:red, 74; green, 74; blue, 74 }  ,fill opacity=1 ] (96,260.5) .. controls (96,122.7) and (207.7,11) .. (345.5,11) .. controls (483.3,11) and (595,122.7) .. (595,260.5) .. controls (595,398.3) and (483.3,510) .. (345.5,510) .. controls (207.7,510) and (96,398.3) .. (96,260.5) -- cycle ;
\draw  [fill={rgb, 255:red, 155; green, 155; blue, 155 }  ,fill opacity=1 ] (164.06,260.5) .. controls (164.06,160.29) and (245.29,79.06) .. (345.5,79.06) .. controls (445.71,79.06) and (526.94,160.29) .. (526.94,260.5) .. controls (526.94,360.71) and (445.71,441.94) .. (345.5,441.94) .. controls (245.29,441.94) and (164.06,360.71) .. (164.06,260.5) -- cycle ;
\draw  [fill={rgb, 255:red, 255; green, 255; blue, 255 }  ,fill opacity=1 ] (234.38,260.5) .. controls (234.38,199.13) and (284.13,149.38) .. (345.5,149.38) .. controls (406.87,149.38) and (456.63,199.13) .. (456.63,260.5) .. controls (456.63,321.87) and (406.87,371.63) .. (345.5,371.63) .. controls (284.13,371.63) and (234.38,321.87) .. (234.38,260.5) -- cycle ;
\draw  [color={rgb, 255:red, 0; green, 0; blue, 0 }  ,draw opacity=1 ][fill={rgb, 255:red, 0; green, 0; blue, 0 }  ,fill opacity=1 ] (289,306) .. controls (289,303.24) and (291.24,301) .. (294,301) .. controls (296.76,301) and (299,303.24) .. (299,306) .. controls (299,308.76) and (296.76,311) .. (294,311) .. controls (291.24,311) and (289,308.76) .. (289,306) -- cycle ;
\draw  [fill={rgb, 255:red, 0; green, 0; blue, 0 }  ,fill opacity=1 ] (340.5,260.5) .. controls (340.5,257.74) and (342.74,255.5) .. (345.5,255.5) .. controls (348.26,255.5) and (350.5,257.74) .. (350.5,260.5) .. controls (350.5,263.26) and (348.26,265.5) .. (345.5,265.5) .. controls (342.74,265.5) and (340.5,263.26) .. (340.5,260.5) -- cycle ;

\draw (284,23.4) node [anchor=north west][inner sep=0.75pt]  [font=\fontsize{2.65em}{3.18em}\selectfont]  {$\mathrm{supp} \ f$};
\draw (268,273.4) node [anchor=north west][inner sep=0.75pt]  [font=\fontsize{2.65em}{3.18em}\selectfont]  {$y$};
\draw (365,245.4) node [anchor=north west][inner sep=0.75pt]  [font=\fontsize{2.65em}{3.18em}\selectfont]  {$y_{0}$};
\draw (284,175.4) node [anchor=north west][inner sep=0.75pt]  [font=\fontsize{2.65em}{3.18em}\selectfont]  {$B_{4}{}_{r}( y_{0})$};
\draw (295,96.4) node [anchor=north west][inner sep=0.75pt]  [font=\fontsize{2.65em}{3.18em}\selectfont]  {$B_{5}{}_{r}( y_{0})$};

\end{tikzpicture}

		}
		\caption{}
		
	\end{figure}
	\begin{proof}
		Choose a shortest geodesic line $\gamma$ connecting $y$ and $y_0$.	Then we have
		\begin{equation*}
			\bigg|\int_0^\infty t^{-\frac12}\int_M \mathcal{K}(y,z,t)f(z) dzdt\bigg|=\bigg|\int_0^\infty t^{-\frac12}\int_M \tau_\gamma(y_0,y;z)\mathcal{K}(y,z,t)f(z) dzdt\bigg|.
		\end{equation*}
            By the definition of $T$, we have the following estimate,
		\begin{equation*}
			\begin{split}
				&\quad||Tf|(y)-|Tf|(y_0)|\\
                &\leq \int_0^\infty t^{-\frac12}\int_M \bigg|\tau_\gamma(y_0,y;z)\mathcal{K}(y,z,t)-\mathcal{K}(y_0,z,t)\bigg||f|(z) dzdt\\
				&= \bigg\{\int_{\rho^{-2}d^2(y_0,y)}^\infty+\int_0^{\rho^{-2}d^2(y_0,y)}\bigg\} t^{-\frac12}
                   \int_M \bigg|\tau_\gamma(y_0,y;z)\mathcal{K}(y,z,t)-\mathcal{K}(y_0,z,t)\bigg||f|(z) dzdt\\
				&=:I+II.
			\end{split}
		\end{equation*}
        
		For the first term $I$, by condition (b) in Definition \ref{def5.1}, we have
	\begin{equation}
        \label{eq5.9}
		\begin{split}
		I&\leq \int_{\rho^{-2} d^2(y_0,y)}^\infty t^{-\frac12}\int_M \bigg|\tau_\gamma(y_0,y;z)\mathcal{K}(y,z,t)-\mathcal{K}(y_0,z,t)\bigg||f|(z) dzdt\\
			& \leq C_2\int_{\rho^{-2}d^2(y_0,y)}^\infty t^{-\frac12}\int_M V_{y_0}^{-1}(\sqrt{t})t^{-\frac{1+\alpha}{2}}d^\alpha(y,y_0)e^{-\frac{d^2(y_0,z)}{C_1t}}|f|(z)dzdt\\
			& \leq C_2\int_{0}^\infty t^{-\frac12}\int_M V_{y_0}^{-1}(\sqrt{t})
            t^{-\frac{1+\alpha}{2}}d^\alpha(y,y_0)e^{-\frac{d^2(y_0,z)}{C_1t}}|f|(z)dzdt. 
		\end{split}
	\end{equation} 	
	Choosing $i_0\in \IN^+$ such that $5^{i_0}r\leq 1<5^{i_0+1}r$ and defining
	\begin{equation*}
		\begin{split}
			A_i:=B_{5^{i+1}r}(x_0)\setminus B_{5^ir}(x_0),
		\end{split}
	\end{equation*}
        then $\textrm{supp}(f) \subset B \subset \cup_{i=2}^{\infty} A_i$. Thus, we can proceed from (\ref{eq5.9}) to obtain
		\begin{equation}
            \label{eq5.11}
			\begin{split}
			I \leq C_2\sum_{i=2}^{i_0}\int_0^\infty t^{-\frac12}\int_{A_i}  V_{y_0}^{-1}(\sqrt{t})t^{-\frac{1+\alpha}{2}}d^\alpha(y,y_0)e^{-\frac{d^2(y_0,z)}{C_1t}}|f|dzdt. 
			\end{split}
		\end{equation}
	 For each $i$, we have 
		\begin{equation}\label{eq5.12}
			\begin{split}
				&\quad\int_0^\infty t^{-\frac12}\int_{A_i}  V_{y_0}^{-1}(\sqrt{t})t^{-\frac{1+\alpha}{2}}d^\alpha(y,y_0)e^{-\frac{d^2(y_0,z)}{C_1t}}|f|\\
				&= \bigg(\int_{ (5^ir)^2}^\infty+\int_{0}^{(5^ir)^2} \bigg)\int_{A_i}  V_{y_0}^{-1}(\sqrt{t})t^{-\frac{1+\alpha}{2}}d^\alpha(y,y_0)e^{-\frac{d^2(y_0,z)}{C_1t}}|f|\\
				&=I_{1,i}+I_{2,i}.
			\end{split}
		\end{equation}
		For $I_{1,i}$, it follows from volume comparison that
		\begin{equation}\label{eq5.13}
			\begin{split}
				&\quad\int_{ (5^ir)^2}^\infty t^{-\frac12}\int_{A_i}  V_{y_0}^{-1}(\sqrt{t})t^{-\frac{1+\alpha}{2}}d^\alpha(y,y_0)e^{-\frac{d^2(y_0,z)}{C_1t}}|f|\\
				&\leq Cr^{\alpha} \cdot \left\{ \int_{ (5^ir)^2}^\infty t^{-\frac{2+\alpha}{2}} \right\} \cdot \left\{ \int_{A_i}V^{-1}_{y_0}(5^{i}r)|f| \right\}\\
				&\leq C5^{-\alpha i}\fint_{B_{5^{i+1}r}(y_0)} |f|
				 \leq C5^{-\alpha i}\mathbf{M}(|f|)(y_0).
			\end{split}
		\end{equation}
		We next estimate $I_{2,i}$.
            Direct calculation yields
		\begin{equation}\label{eq5.14}
			\begin{split}
				&\quad\int_{0}^{(5^ir)^2}t^{-\frac{1}{2}} \int_{A_i}  V_{y_0}^{-1}(\sqrt{t})t^{-\frac{1+\alpha}{2}}d^\alpha(y,y_0)e^{-\frac{d^2(y_0,z)}{C_1t}}|f|\\
				&\leq Cr^\alpha\int_{0}^{(5^ir)^2}t^{-\frac{2+\alpha}{2}}\int_{A_i} V^{-1}_{y_0}(5^ir)\bigg(\frac{5^ir}{\sqrt{t}}\bigg)^ne^{-\frac{5^{2(i-1)}r^2}{C_1t}}|f|\\
				&\leq Cr^\alpha\int_{0}^{(5^ir)^2}t^{-\frac{2+\alpha}{2}}\int_{B_{5^{i+1}r}(y_0)} V^{-1}_{y_0}(5^{i+1}r)\bigg(\frac{5^ir}{\sqrt{t}}\bigg)^ne^{-\frac{5^{2(i-1)}r^2}{C_1t}}|f|\\
				&\leq C5^{-\alpha i}\int_{0}^{(5^ir)^2}t^{-1}\bigg(\frac{5^ir}{\sqrt{t}}\bigg)^{n+\alpha} e^{-\frac{5^{2(i-1)}r^2}{C_1t}}dt\fint_{B_{5^{i+1}r}(y_0)}|f|\\
				&\leq C5^{-\alpha i}\mathbf{M}(|f|)(y_0),
			\end{split}
		\end{equation}
		where we use
		\begin{equation*}
			\begin{split}
				\int_{0}^{(5^ir)^2}t^{-1}\bigg(\frac{5^ir}{\sqrt{t}}\bigg)^{n+\alpha} e^{-\frac{5^{2(i-1)}r^2}{C_1t}}dt
				=2\int_{0}^1s^{-(n+1+\alpha)}e^{-\frac{5^{-2}}{C_1s^2}}ds
				\leq C.
			\end{split}
		\end{equation*}
		Plugging \eqref{eq5.13}, \eqref{eq5.14} into \eqref{eq5.12}, we obtain 
		\begin{equation}
			\int_0^\infty t^{-\frac12}\int_{A_i}  V_{y_0}^{-1}(\sqrt{t})t^{-\frac{1+\alpha}{2}}d^\alpha(y,y_0)e^{-\frac{d^2(y_0,z)}{C_1t}}|f|\leq C5^{-\alpha i}\mathbf{M}(|f|)(y_0).
		\end{equation}
		Hence, we have
		\begin{equation}\label{eq5.17}
			\begin{split}
				I\leq \sum_{i=2}^{i_0}C5^{-\alpha i}\mathbf{M}(|f|)(y_0)\leq C\mathbf{M}(|f|)(y_0).
			\end{split}
		\end{equation}
		
		For the second term $II$, we directly have
		\begin{equation}
        \label{eq5.18}
			\begin{split}
				II
				& \leq 	\int_0^{\rho^{-2}d^2(y_0,y)} t^{-\frac12}\int_M (|\mathcal{K}(y,z,t)|+|\mathcal{K}(y_0,z,t)|)|f|(z) dzdt\\
				&	\leq C_1	\int_0^{\rho^{-2}d^2(y_0,y)} t^{-\frac12}
                \int_M	\left\{V^{-1}_{y_0}(\sqrt{t}) + V^{-1}_{y}(\sqrt{t})\right\} e^{-C_0C_1t}e^{-\frac{d^2(y_0,z)}{C_1t}}|f|(z) dzdt. 
			\end{split}
		\end{equation}
		For any $z\in \mathrm{supp}f\subset B\setminus B_{5r}$, we have $d(y,z) \geq r$ by triangle inequality. Thus, 
        \begin{align}
            d(y_0,z) \leq d(y_0, y)+d(y,z) \leq 4r + d(y,z) \leq 5 d(y, z). 
        \label{eqn:HC17_5}    
        \end{align}
        By volume comparison, 
		\begin{equation}
        \label{eqn:HC17_6}
			\begin{split}
				V^{-1}_{y}(\sqrt{t})\leq CV^{-1}_{y}(\sqrt{t}+d(y_0,y))\bigg(1+\frac{d(y_0,y)}{\sqrt{t}}\bigg)^n\leq CV^{-1}_{y_0}(\sqrt{t})\bigg(1+\frac{d(y_0,y)}{\sqrt{t}}\bigg)^n.
			\end{split}
		\end{equation}
		Plugging (\ref{eqn:HC17_5}) and (\ref{eqn:HC17_6}) into \eqref{eq5.18}, we obtain 
		\begin{equation}
        \label{eqn:HC17_1}
			\begin{split}
				II&	\leq C	\int_0^{\rho^{-2}d^2(y_0,y)} t^{-\frac12}\int_M	V^{-1}_{y_0}(\sqrt{t})\bigg(1+\frac{d(y_0,y)}{\sqrt{t}}\bigg)^ne^{-C_0C_1t}e^{-\frac{d^2(y_0,z)}{25C_1t}}|f|(z) dzdt\\
					&\leq C	\int_0^{\rho^{-2}d^2(y_0,y)} t^{-\frac12}\int_M	V^{-1}_{y_0}(\sqrt{t})\bigg(\frac{d(y_0,y)}{\sqrt{t}}\bigg)^ne^{-\frac{d^2(y_0,z)}{25C_1t}}|f|(z) dzdt\\
				&=C	\sum_{i=2}^{i_0}\int_0^{\rho^{-2}d^2(y_0,y)} t^{-\frac12}\int_{A_i}	V^{-1}_{y_0}(\sqrt{t})\bigg(\frac{d(y_0,y)}{\sqrt{t}}\bigg)^ne^{-\frac{d^2(y_0,z)}{25C_1t}}|f|(z) dzdt\\
				&=:C\sum_i II_i.
			\end{split}
		\end{equation}
		For each $i$,
			\begin{equation}
            \label{eqn:HC17_2}
			\begin{split}
			II_i&=\int_0^{\rho^{-2}d^2(y_0,y)} t^{-\frac12}\int_{A_i}	V^{-1}_{y_0}(\sqrt{t})\bigg(\frac{d(y_0,y)}{\sqrt{t}}\bigg)^ne^{-\frac{5^{2(i-1)}r^2}{25C_1t}}|f|(z) dzdt\\
			&\leq C\int_0^{\rho^{-2}d^2(y_0,y)}  t^{-\frac12}\int_{A_i}	V^{-1}_{y_0}(5^{i+1}r)\bigg(\frac{5^{i}r d(y_0,y)}{t}\bigg)^ne^{-\frac{5^{2(i-1)}r^2}{25C_1t}}|f|(z) dzdt\\
			& \leq C \cdot \left\{ \int_0^{\rho^{-2}d^2(y_0,y)}  t^{-\frac12}	\bigg(\frac{5^{i}r d(y_0,y)}{t}\bigg)^ne^{-\frac{5^{2(i-1)}r^2}{25C_1t}}dt \right\} 
            \cdot \mathbf{M}(|f|)(y_0).
			\end{split}
		\end{equation}
		Since $d(y_0,y)\leq 4r$, by letting  $s=5^{i}r/\sqrt{t}$, we have
		\begin{equation}
        \label{eqn:HC17_3}
			\begin{split}
				&\quad\int_0^{\rho^{-2}d^2(y_0,y)}  t^{-\frac12}	\bigg(\frac{5^{i}r d(y_0,y)}{t}\bigg)^ne^{-\frac{5^{2(i-1)}r^2}{25C_1t}}dt\\
			    &\leq C\int_0^{16\rho^{-2}r^2}  t^{-\frac12}	\bigg(\frac{5^{i}r^2 }{t}\bigg)^ne^{-\frac{5^{2(i-1)}r^2}{25C_1t}}dt\\
			    &\leq C 5^{-(n-1)i}r\int_{0}^\infty  s^{2(n-1)}e^{-\frac{s^2}{625C_1}}ds
			    \leq C 5^{-(n-1)i}.
			\end{split}
		\end{equation}
		Combining (\ref{eqn:HC17_1}), (\ref{eqn:HC17_2}) and (\ref{eqn:HC17_3}), we arrive at 
	\begin{equation}
    \label{eq5.23}
		\begin{split}
			II\leq C \cdot \left\{\sum_{i=2}^{i_0}5^{-(n-1)i} \right\} \cdot \mathbf{M}(|f|)(y_0)\leq C\mathbf{M}(|f|)(y_0).
		\end{split}
		\end{equation}
		Then (\ref{eqn:HC17_4}) follows from the combination of~\eqref{eq5.17} and~\eqref{eq5.23}.
	\end{proof}

	\begin{lemma}
    \label{lm5.5}
 The same assumptions as in Lemma~\ref{lm5.4}. 
 Then there exists a positive constant $C=C(n,\rho,\alpha,\Lambda_0,C_0,C_1,C_2)$ such that the following property holds.
        
        If $\mathbf{M}(|Tf|^2)(y_0)\leq a^2$ and $\mathbf{M}(|f|^2)(y_0)\leq b^2$, then
		\begin{equation}
        \label{eqn:HC17_10}
			\begin{split}
				\mathbf{M}(|Tf|^2)(y)\leq C(a^2+b^2) 
			\end{split}
		\end{equation}
        for all $y \in B_{3r}(y_0)\cap 2B$. 
	\end{lemma}
	
		\begin{proof}
		By H\"older inequality, 
		\begin{equation*}
			\begin{split}
				\mathbf{M}(|Tf|)(y_0)&\leq \sqrt{\mathbf{M}(|Tf|^2)(y_0)}\leq a, \\
				\mathbf{M}(|f|)(y_0)&\leq \sqrt{\mathbf{M}(|f|^2)(y_0)}\leq b.
			\end{split}
		\end{equation*}
        
	For $y\in B_{4r}(y_0)$, by Lemma \ref{lm5.4}, we have
	\begin{equation*}
		\begin{split}
			|Tf|(y)&\leq |Tf|(y_0)+||Tf|(y)-|Tf|(y_0)|\\
			&\leq \mathbf{M}(|Tf|)(y_0)+||Tf|(y)-|Tf|(y_0)|\\
			&\leq  a+Cb.
		\end{split}
	\end{equation*}

			For any $y\in B_{3r}(y_0)\cap 2B$, we now want to estimate
            \begin{align*}
               \sup_{0<s<4} \fint_{B_s(y)} |Tf|^2. 
            \end{align*}
        If $s\leq r$, we have
		\begin{equation}
			\fint_{B_s(y)}|Tf|^2\leq \sup_{B_{s}(y)} |Tf|^2 \leq \sup_{B_{4r}(y_0)} |Tf|^2 \leq C(a^2+b^2).
        \label{eqn:HC17_7}    
		\end{equation}
	    If $ r\leq s\leq \frac12$, $d(y,y_0)\leq 4s$. Thus, 
		\begin{equation}
			\fint_{B_s(y)}|Tf|^2\leq \frac{|B_{5s}(y_0)|}{|B_{5}(y)|}\fint_{B_{3s}(y_0)}|Tf|^2\leq Ca^2.
        \label{eqn:HC17_8}    
		\end{equation}
        If $\frac12\leq s\leq 4$, by \eqref{eq5.3}
        \begin{equation}
        \begin{split}
            \fint_{B_s(y)}|Tf|^2&\leq \frac{1}{|B_{\frac12}(y)|}\int_{B_{5s}(y_0)}|Tf|^2
             \leq \frac{C}{|B_{2}(y_0)|}\int_{B_{2}(y_0)}|f|^2\\
            &\leq C\mathbf{M}(|f|^2)(y_0) \leq C b^2.
        \end{split}
        \label{eqn:HC17_9}
		\end{equation}
        It follows from the combination of (\ref{eqn:HC17_7}), (\ref{eqn:HC17_8}) and (\ref{eqn:HC17_9}) that 
        \begin{align*}
            \mathbf{M}(|Tf|^2)(y)=\sup_{0<s<4} \fint_{B_s(y)} |Tf|^2 \leq C(a^2+b^2),  
        \end{align*}
        which is exactly (\ref{eqn:HC17_10}). 
	\end{proof}

	\begin{lemma}\label{lm5.6}
		Let $x_0\in 2B$ and $f\in L^2(B)$. Then for any $0<r\leq 4$, there exists a positive constant $N=N(n,\rho,\alpha,\Lambda_0,C_0,C_1,C_2,C_3)$ with the following property.

        For any $\varepsilon\in (0,1)$, we can choose a $\delta>0$ such that if 
		\begin{equation}\label{eq5.29}
			\begin{split}
				|\{x\in B_{5r}(x_0) \cap 2B:\mathbf{M}(|Tf|^2)(x)\leq 1, \; \mathbf{M}(|f|^2)(x)\leq \delta^2\}|
				\geq \frac{1}{2}|B_{5r}(x_0)\cap 2B|, 
			\end{split}
		\end{equation}
		then 
		\begin{equation*}
			\begin{split}
				|\{x\in B_{5r}(x_0) \cap 2B:\mathbf{M}(|Tf|^2)(x)> N^2\}|
                \leq \varepsilon |B_{5r}(x_0)\cap 2B|. 
			\end{split}
		\end{equation*}
	\end{lemma}
	
	\begin{proof}
		Define $f_1=f \chi_{B_{25r}(x_0)}$ where $\chi$ is the character function of a set.
        Then $f_1(x)=f(x)$ if $x \in B_{25r}(x_0)$ and $f_1(x)=0$ otherwise.
        Define $f_2=f-f_1$.
        By \eqref{eq5.29}, there exists a point $x_1\in B_{5r}(x_0)\cap 2B$ such that  
		\begin{equation*}
			\begin{split}
				\mathbf{M}(|f_1|^2)(x_1)\leq \mathbf{M}(|f|^2)(x_1)\leq \delta^2 .
			\end{split}
		\end{equation*}
		It follows that (cf. Remark \ref{rmk2.3})
		\begin{equation}
        \label{eqn:HC18_1}
			\begin{split}
				\int_{B_{30r}(x_1)} |f_1|^2 \leq  |B_{30r}(x_1)| \cdot \mathbf{M}(|f_1|^2)(x_1)
                \leq  |B_{30r}(x_1)| \delta^2. 
			\end{split}
		\end{equation}
		Since $\mathbf{M}$ is of weak type $(1,1)$ and $T$ is of strong type $(2,2)$, we have
		\begin{equation*}
			\begin{split}
				&\quad |\{x\in B_{5r}(x_0) \cap 2B: \mathbf{M}(|Tf_1|^2)(x)>1\}| \\
                &\leq C \|Tf_1\|^2_{L^2(B_{30r}(x_1))}\leq C\|f_1\|^2_{L^2(B_{30r}(x_1))}.
			\end{split}
		\end{equation*}
		Thus, by \eqref{eqn:HC18_1},
		\begin{equation}
			\begin{split}
				|\{x\in B_{5r}(x_0)\cap 2B:\mathbf{M}(|Tf_1|^2)(x)>1\}|\leq C\delta^2 |B_{30r}(x_1)|.
			\end{split}
		\end{equation}
		Thanks to Corollary \ref{coro2.2}, we can choose $\delta$ so small such that
		\begin{equation}\label{eq5.36}
			\begin{split}
				|\{x\in B_{5r}(x_0)\cap 2B:\mathbf{M}(|Tf_1|^2)(x)>1\}|
                <\frac12|B_{5r}(x_0)\cap 2B|.
			\end{split}
		\end{equation}
		From \eqref{eq5.29}  and \eqref{eq5.36}, there exists $x_2\in B_{5r}(x_0)\cap 2B$ such that
		\begin{equation*}
			\begin{split}
				\mathbf{M}(f^2)(x_2)\leq \delta^2,\quad 	\mathbf{M}(|Tf|^2)(x_2)\leq 1, \quad \mathbf{M}(|Tf_1|^2)(x_2)\leq 1.
			\end{split}
		\end{equation*}
		On the other hand, 
		\begin{equation*}
			\begin{split}
				\mathbf{M}(|f_2|^2)(x_2)\leq \mathbf{M}(f^2)(x_2)\leq \delta^2.
			\end{split}
		\end{equation*}
	    Note that 
		\begin{equation*}
			\begin{split}
				|Tf_2|^2=|Tf-Tf_1|^2\leq 2(|Tf|^2+|Tf_1|^2).
			\end{split}
		\end{equation*}
		By the sub-additive of $\mathbf{M}$, we have 
		\begin{equation*}
			\begin{split}
				\mathbf{M}(|Tf_2|^2)(x_2)\leq 2(\mathbf{M}(|Tf|^2)(x_2)+\mathbf{M}(|Tf_1|^2)(x_2))\leq 4.
			\end{split}
		\end{equation*}
		Applying Lemma \ref{lm5.5} with $x_2$ and radius $4r$, and noting that $B_{5r}(x_0) \subset B_{12r}(x_2)$, we have 
        \begin{equation*}
          \mathbf{M}(|Tf_2|^2)(x)
          \leq C \left\{\mathbf{M}(|f_2|^2)(x_2) + \mathbf{M}(|Tf_2|^2)(x_2)\right\}
          \leq C(4+\delta^2) \leq 5C
        \end{equation*}
        for all $x \in B_{5r}(x_0)\cap 2B$. 
        Fix the last $C$ in the above inequality and choose
        \begin{align*}
            N^2 >\mathrm{max}\{5^n, 5C\},  
        \end{align*}
        where $N$ is a large number whose exact value will be determined later (cf. (\ref{eqn:HC18_2})). 
        Then we have
		\begin{equation*}
			\begin{split}
				\mathbf{M}(|Tf|^2)(x)
                &\leq 2 \left\{\mathbf{M}(|Tf_1|^2)(x)+\mathbf{M}(|Tf_2|^2)(x) \right\}\\
                &\leq 2 \tilde{C}^2 + 2\mathbf{M}(|Tf_1|^2)(x)
                \leq \frac{N^2}{2} + 2\mathbf{M}(|Tf_1|^2)(x), 
			\end{split}
		\end{equation*}
        for any $x\in B_{5r}(x_0)\cap 2B$. Rewriting the above inequality as
        \begin{align*}
            \mathbf{M}(|Tf_1|^2)(x)> \frac12 \mathbf{M}(|Tf|^2)(x) -\frac{N^2}{4},  
        \end{align*}
        it is clear that 
		\begin{equation*}
			\begin{split}
				 \{x\in B_{5r}(x_0)\cap 2B:\mathbf{M}(|Tf|^2)(x)>N^2\}
				 \subset \{x\in B_{5r}(x_0)\cap 2B:\mathbf{M}(|Tf_1|^2)(x)>N^2/4\}.
			\end{split}
		\end{equation*}
		Applying the weak type $(1,1)$ inequality of $\mathbf{M}$, the volume comparison
        and inequality (\ref{eqn:HC18_1}), we have 
		\begin{equation*}
			\begin{split}
				&\quad	|\{x\in B_{5r}(x_0)\cap 2B:\mathbf{M}(|Tf|^2)(x)>N^2\}|\\
				&\leq |\{x\in B_{5r}(x_0)\cap 2B:\mathbf{M}(|Tf_1|^2)(x)>N^2/4\}|\\
				&\leq C\frac{\|Tf_1\|^2_{L^2}}{N^2/4}\leq \frac{4CC_2}{N^2}\|f_1\|^2_{L^2} 
                \leq  \frac{4CC_2 \cdot |B_{30r}(x_1)|}{N^2}  \delta^2.  
			\end{split}
		\end{equation*}
        By volume comparison, $|B_{30r}(x_1)| \leq |B_{120}(x_1)|$ is bounded by a uniform constant $C'$ depending only on $n,\Lambda_0$.  Thus, we can choose a uniform $N$ sufficiently large such that
        \begin{equation}
         \label{eqn:HC18_2}
			\begin{split}
				|\{x\in B_{5r}(x_0)\cap 2B:\mathbf{M}(|Tf|^2)(x)>N^2\}|
                \leq  \frac{4CC_2 \cdot |B_{30r}(x_1)|}{N^2}  \delta^2
                <\delta^2. 
			\end{split}
		\end{equation}
		 The proof is complete by setting $\delta=\sqrt{\varepsilon}$. 
	\end{proof}
	
	\begin{lemma}\label{lm5.7}
		Let $\varepsilon, \delta$ be given by Lemma \ref{lm5.6}. Suppose that
		\begin{equation*}
			\begin{split}
				|\{x\in 2B:\mathbf{M}(|Tf|^2)(x)\geq N^2\}|\leq \varepsilon|2B|.
			\end{split}
		\end{equation*}
		Then 
		\begin{equation*}
			\begin{split}
				&\quad |\{x\in 2B:\mathbf{M}(|Tf|^2)(x)> N^2\}|\\
                &\leq \varepsilon_1 \bigg(|\{x\in 2B:\mathbf{M}(|Tf|^2)(x)>1 \}|
				   +|\{x\in 2B:\mathbf{M}(|f|^2)(x)> \delta^2\}|	\bigg)\\
			\end{split}
		\end{equation*}
		where $\varepsilon_1=C(\Lambda_0,n)\varepsilon$.
	\end{lemma}
	
	\begin{proof}	
		Define
		\begin{align*}
			&U:=\{x\in 2B: \mathbf{M}(|Tf|^2)(x)>N^2\}, \\
			&V:=\{x\in 2B: \mathbf{M}(|Tf|^2)(x)>1\}\cup \{x\in 2B: \mathbf{M}(|f|^2)(x)>\delta^2\}.
		\end{align*}
		Since $|U|\leq \varepsilon |2B|$, we see that  for almost  every $x\in U$, there exists an $r_x<4$ such that $|U\cap B_{r_x}(x)|=\varepsilon |B_{r_x}(x)\cap 2B|$ and 
		\begin{equation*}
			|U\cap B_{r}(x)|<\varepsilon |B_{r}(x)\cap 2B|
		\end{equation*}
		for all $r_x<r\leq 4$. By Vitali's covering lemma, there exist countably many disjoint balls $\{B_{r_{x_k}}(x_k)\}$ such that $\cup_{k}B_{5r_{x_k}}(x_k)\cap 2B\supset U$.
		
		By the choice of $B_{r_x}(x)$ and applying Lemma \ref{lm5.6}, we have 
		\begin{equation}\label{eq5.51}
			|V \cap B_{r_{x_k}}(x_k)| \geq \frac{1}{2} |B_{r_{x_k}}(x_k) \cap 2B|.
		\end{equation}
		It follows that
		\begin{equation}\label{eq5.52}
			\begin{split}
				|U| \leq \sum_k |U \cap B_{5r_{x_k}}(x_k)| 
				\leq \varepsilon \sum_k |B_{5r_{x_k}}(x_k) \cap 2B|.
			\end{split}
		\end{equation}
		By volume comparison (cf. Corollary \ref{coro2.2}), there exists a constant $C > 0$ such that
		\begin{equation}\label{eq5.53}
			|B_{5r_{x_k}}(x_k) \cap 2B| \leq C |B_{r_{x_k}}(x_k) \cap 2B|.
		\end{equation}
		Plugging \eqref{eq5.51} and \eqref{eq5.53} into \eqref{eq5.52}, we obtain
		\begin{equation*}
			\begin{split}
				|U| \leq C \varepsilon \sum_k |B_{r_{x_k}}(x_k) \cap 2B| 
				\leq 2C \varepsilon \sum_k |B_{r_{x_k}}(x_k) \cap V|.
			\end{split}
		\end{equation*}
		Since the balls $\{B_{r_{x_k}}(x_k)\}_k$ are pairwise disjoint, it follows that 
		\begin{equation*}
			\begin{split}
				|U| \leq 2C \varepsilon \left| \left( \bigcup_k B_{r_{x_k}}(x_k) \right) \cap V \right|
				\leq \varepsilon_1 |V|,
			\end{split}
		\end{equation*}
		where we choose $\varepsilon_1 = 2C \varepsilon$.
	\end{proof}
	We then establish the good-$\lambda$ inequality, a technique originating from the seminal work of Burkholder and Gundy \cite{BG70} (see also Coifman and Fefferman \cite{CF74}).
    
	\begin{corollary}
        \label{coro5.8}
		For  $\varepsilon_1=2C\varepsilon$ and $\lambda\geq 1$, if $\|f\|^2_{L^2}\leq C^{-1}\varepsilon N^2|2B|$, then the following holds
		\begin{equation}
        \label{eqn:HC18_3}
			\begin{split}
				&\quad |\{x\in 2B:\mathbf{M}(|Tf|^2)(x)>\lambda^2 N^2\}|\\
                &\leq \varepsilon_1
                \bigg(|\{x\in 2B:\mathbf{M}(|Tf|^2)(x)>\lambda^2 \}|
				+|\{x\in 2B:\mathbf{M}(|f|^2)(x)>\lambda^2 \delta^2\}|	\bigg). 
			\end{split}
		\end{equation}
	\end{corollary}
	
	\begin{proof}
    Changing $f$ to $f/\lambda$ if necessary, we can always assume $\lambda=1$.
	Since $\mathbf{M}$ is of weak type $(1,1)$ and $T$ is of strong type $(2,2)$,   we have
		\begin{equation*}
			|\{x\in 2B:\mathbf{M}(|Tf|^2)(x)\geq N^2\}|\leq \frac{C}{N^{2}}\int_{M}|Tf|^2
            \leq \frac{C}{N^{2}}\int_{M}|f|^2. 
		\end{equation*}
		By the condition, we have
		\begin{equation*}
			\begin{split}
				&|\{x\in 2B:\mathbf{M}(|Tf|^2)(x)\geq N^2\}|\leq \varepsilon|2B|.
			\end{split}
		\end{equation*}
        Then Lemma \ref{lm5.7} applies and (\ref{eqn:HC18_3}) follows directly. 
	\end{proof}

	\begin{proof}[Proof of Proposition \ref{prop5.3}:]
		Define
		\begin{equation}
			\tilde{f}:=\frac{f}{\omega},  \quad \omega:=\sqrt{C\varepsilon^{-1}|2B|^{-1}}N^{-1}\|f\|_{L^2}.
        \label{eqn:HC21_1}    
		\end{equation}
        The layer cake representation implies 
        \begin{align}
          \int_{2B}(\mathbf{M}(|T\tilde{f}|^2)(x))^{\frac{p}{2}}
		 =\frac{p}{2}\int_0^\infty t^{\frac{p}{2}-1}|\{x\in  2B:\mathbf{M}(|T\tilde{f}|^2)(x)>t\}|dt.   
        \label{eqn:HC20_1}
        \end{align}
        Separating $(0,\infty)$ as $(0, N) \cup [N, \infty)$ and changing variable, we have 
		\begin{equation}
        \label{eqn:HC20_2}
			\begin{split}
				&\quad  \int_0^\infty t^{\frac{p}{2}-1}|\{x\in  2B:\mathbf{M}(|T\tilde{f}|^2)(x)>t\}|dt\\
				&=N^p \left\{ \int_1^\infty +\int_0^1 \right\} t^{\frac{p}{2}-1}|\{x\in  2B:\mathbf{M}(|T\tilde{f}|^2)(x)>N^2t\}|dt\\
                &\leq N^p  \int_1^\infty t^{\frac{p}{2}-1}|\{x\in  2B:\mathbf{M}(|T\tilde{f}|^2)(x)>N^2t\}|dt +N^p |2B|.
			\end{split}
		\end{equation}
		By Corollary \ref{coro5.8}, we have
		\begin{equation*}
			\begin{split}
				& \quad \int_1^\infty t^{\frac{p}{2}-1}|\{x\in  2B:\mathbf{M}(|T\tilde{f}|^2)(x)>N^2t\}|dt\\
				&\leq \varepsilon_1 \int_1^\infty t^{\frac{p}{2}-1} \bigg(|\{x\in  2B:\mathbf{M}(|T\tilde{f}|^2)(x)>t \}|
				  +|\{x\in  2B:\mathbf{M}(|\tilde{f}|^2)(x)>t \delta^2\}|	\bigg)dt\\
				&\leq \varepsilon_1 \int_1^\infty t^{\frac{p}{2}-1} |\{x\in  2B:\mathbf{M}(|T\tilde{f}|^2)(x)>t \}|dt\\
				&\quad +\frac{\varepsilon_1}{\delta^p}\int_{\delta^2}^\infty t^{\frac{p}{2}-1}	|\{x\in 2B:\mathbf{M}(|\tilde{f}|^2)(x)>t \}|dt\\
                &\leq \varepsilon_1 \int_{0}^\infty t^{\frac{p}{2}-1} |\{x\in  2B:\mathbf{M}(|T\tilde{f}|^2)(x)>t \}|dt\\
				&\quad +\frac{\varepsilon_1}{\delta^p}\int_{0}^\infty t^{\frac{p}{2}-1}	|\{x\in 2B:\mathbf{M}(|\tilde{f}|^2)(x)>t \}|dt. 
			\end{split}
		\end{equation*}
		Plugging the above inequality and (\ref{eqn:HC20_2}) into (\ref{eqn:HC20_1}), we obtain
		\begin{equation*}
			\begin{split}
				&\quad \frac{p}{2}\int_0^\infty t^{\frac{p}{2}-1}|\{x\in  2B:\mathbf{M}(|T\tilde{f}|^2)(x)>t\}|dt\\
				&\leq(\varepsilon_1N^p)\frac{p}{2}\int_0^\infty t^{\frac{p}{2}-1} 
                |\{x\in  2B:\mathbf{M}(|T\tilde{f}|^2)(x)>t \}|dt\\
				&\quad +(\varepsilon_1N^p)\frac{p}{2 \delta^p} 
                \|\mathbf{M} (|\tilde{f}|^2)\|_{L^p(2B)}+\frac{p}{2}N^p|2B|.
			\end{split}
		\end{equation*}
        Choosing $\varepsilon$ sufficiently small such that  $\varepsilon_1<\frac{1}{2N^p}$, we have
		\begin{equation}
        \label{eqn:HC21_4}
			\begin{split}
				&\quad \frac{p}{2}\int_0^\infty t^{\frac{p}{2}-1}|\{x\in  2B:\mathbf{M}(|T\tilde{f}|^2)(x)>t\}|dt\\
				& \leq \frac{p}{2\delta^p}\|\mathbf{M} (|\tilde{f}|^2)\|_{L^{\frac{p}{2}}(2B)}+pN^p|2B|\\
				& \leq \frac{Cp}{\delta^p}\| \tilde{f}\|_{L^p(2B)}+pN^p|2B|, 
			\end{split}
		\end{equation}
		where we use the strong type $(\frac{p}{2},\frac{p}{2})$ inequality of $\mathbf{M}$.
        Recall (\ref{eqn:HC20_1}). It follows from (\ref{eqn:HC21_4}) that
        \begin{align*}
          \int_{2B}(\mathbf{M}(|T\tilde{f}|^2)(x))^{\frac{p}{2}}
          \leq \frac{C}{\delta^p}\| \tilde{f}\|_{L^p(2B)}+pN^p|2B|. 
        \end{align*}
		Since $\tilde{f}$ is a rescaling of $f$ by (\ref{eqn:HC21_1}), the above inequality is equivalent to 
		\begin{equation}\label{eq5.68}
			\begin{split}
				\int_{2B}(\mathbf{M}(|Tf|^2)(x))^{\frac{p}{2}}\leq \frac{2Cp}{\delta^p}\int_{2B}|f|^p+Cp\varepsilon^{-\frac{p}{2}}|2B|^{1-\frac{p}{2}}\|f\|^{p}_{L^2}.
			\end{split}
		\end{equation}
		The H\"older inequality implies that
		\begin{equation*}
			\begin{split}
				|2B|^{1-\frac{p}{2}}\|f\|^{p}_{L^2}=\bigg(|2B|^{\frac{2}{p}-1}\int_{2B}|f|^2\bigg)^{\frac{p}{2}}\leq \int_{2B} |f|^p.
			\end{split}
		\end{equation*}
		Plugging it into \eqref{eq5.68} and absorbing $p,\epsilon$ and $\delta$ into $C$, we arrive at 
		\begin{equation*}
			\int_{2B}(\mathbf{M}(|Tf|^2)(x))^{\frac{p}{2}}\leq C\int_{B} |f|^p, 
		\end{equation*}
		which  yields  (\ref{eqn:HC21_5}). 
	\end{proof}
    
	Next, we estimate $\|Tf\|_{L^p(M\setminus 2B)}$.
    
	\begin{proposition}\label{lm5.9}
		For $1\leq p<\infty$ and $f\in L^p(\Lambda^jT^*M)$ with $\mathrm{supp}f\subset B$, we have
		\begin{equation}
			\int_{M\setminus 2B}|Tf|^p\leq C\int_{B}|f|^p. 
         \label{eqn:HC18_4}   
		\end{equation}
	\end{proposition}
	\begin{proof}
	Define $B^i := 2^{i+1}B \setminus 2^iB$. By direct calculation, the heat kernel upper bounds imply
	\begin{equation}
    \label{eq5.73}
		\begin{split}
			&\quad\int_{M\setminus 2B}\bigg|\int_0^\infty t^{-\frac{1}{2}}\int_{B}\mathcal{K}(x,y,t)f(y)dydt\bigg|^p dx \\
			&\leq C\int_{M\setminus 2B}\bigg(\int_0^\infty \int_{B}V^{-1}_x(\sqrt{t})t^{-1}e^{-C_0C_1t}e^{-\frac{d^2(x,y)}{C_1t}}|f(y)|dydt\bigg)^p dx \\
			& = C\sum_{i=1}^{\infty}\int_{B^i}\bigg(\int_0^\infty \int_{B}V^{-1}_x(\sqrt{t})t^{-1}e^{-C_0C_1t}e^{-\frac{d^2(x,y)}{C_1t}}|f(y)|dydt\bigg)^p dx.
		\end{split}
	\end{equation}
	For any $x \in B^i$ and $y \in B$, we have $d(x,y) \geq 2^{i-1}$. Hence,
	\begin{equation}
    \label{eq5.74}
		\begin{split}
			&\quad\int_0^\infty\int_{B}V^{-1}_x(\sqrt{t})t^{-1}e^{-C_0C_1t}e^{-\frac{d^2(x,y)}{C_1t}}|f(y)|dydt \\
			&\leq \int_0^\infty\int_{B}V^{-1}_x(\sqrt{t})t^{-1}e^{-C_0C_1t}e^{-\frac{2^{2i}}{4C_1t}}|f(y)|dydt \\
			&=\left\{\int_0^\infty V^{-1}_x(\sqrt{t})t^{-1}e^{-C_0C_1t}e^{-\frac{2^{2i}}{4C_1t}}dt \right\} \int_B |f| dy.
		\end{split}
	\end{equation}
    
    We claim that
    \begin{align}
    \label{eqn:HC21_2}
       \int_0^\infty V^{-1}_x(\sqrt{t})t^{-1}e^{-C_0C_1t}e^{-\frac{2^{2i}}{4C_1t}}dt
       \leq C|2^{i+1}B|^{-1}2^{-2i}. 
    \end{align}
	In fact, for $t \leq 2^{2(i+2)}$, by volume comparison and the fact that $B_{2^{i+1}}(q) \subset B_{2^{i+2}}(x)$, we obtain
	\begin{equation*}
		V^{-1}_x(\sqrt{t}) \leq Ce^{\sqrt{\Lambda_0}2^{i+2}} \left( \frac{t}{2^{2i}} \right)^{-n/2} V^{-1}_x(2^{i+2}) \leq Ce^{\sqrt{\Lambda_0}2^{i+2}} \left( \frac{t}{2^{2i}} \right)^{-n/2} |2^{i+1}B|^{-1}.
	\end{equation*}
	For $t \geq 2^{2(i+2)}$, we simply have
	\begin{equation*}
		V^{-1}_x(\sqrt{t}) \leq V^{-1}_x(2^{i+2}) \leq |2^{i+1}B|^{-1}.
	\end{equation*}
	Consequently, for the small time interval
	\begin{equation}\label{eq5.77}
		\begin{split}
			&\quad\int_0^{2^{2(i+2)}} V^{-1}_x(\sqrt{t})t^{-1}e^{-C_0C_1t}e^{-\frac{2^{2i}}{4C_1t}}dt \\
			& \leq C|2^{i+1}B|^{-1}2^{-2i} \int_0^{2^{2(i+2)}} \left( \frac{t}{2^{2i}} \right)^{-(n+2)/2} \exp \left\{ \sqrt{\Lambda_0}2^{i+2} - C_0C_1t - \frac{2^{2i}}{4C_1t} \right\} dt \\
			& \leq C|2^{i+1}B|^{-1}2^{-2i}.
		\end{split}
	\end{equation}
	By the Cauchy-Schwarz inequality,
	\begin{equation*}
		\frac{1}{2}C_0C_1t + \frac{2^{2i}}{8C_1t} \geq \frac{1}{2}\sqrt{C_0}2^i \geq \sqrt{\Lambda_0}2^{i+2}, 
	\end{equation*}
    we have
	\begin{equation*}
		\begin{split}
			& \quad\int_0^{2^{2(i+2)}} \left( \frac{t}{2^{2i}} \right)^{-(n+2)/2} \exp \left\{ \sqrt{\Lambda_0}2^{i+2} - C_0C_1t - \frac{2^{2i}}{4C_1t} \right\} dt \\
			& \leq  \int_0^{2^{2(i+2)}} \bigg(\frac{t}{2^{2i}}\bigg)^{-(n+2)/2}e^{-\frac{2^{2i}}{8C_1t}}e^{-\frac12C_0C_1t} dt\\
			& \leq C\int_0^{2^{2(i+2)}} e^{-\frac12C_0C_1t}dt\leq C, 
		\end{split}
	\end{equation*}
	where we use the fact 
    \begin{align*}
        \sup_{s>0} s^{-(n+2)/2}e^{-\frac{1}{8C_1s}} < C. 
    \end{align*}
    Thus, for the large time interval we have
	\begin{equation}
    \label{eq5.80}
		\begin{split}
			&\quad \int_{2^{2(i+2)}}^\infty V^{-1}_x(\sqrt{t})t^{-1}e^{-C_0C_1t}e^{-\frac{2^{2i}}{4C_1t}}dt\\ 
			&\leq |2^{i+1}B|^{-1} \int_{2^{2(i+2)}}^\infty t^{-1}e^{-C_0C_1t} dt 
			 \leq C|2^{i+1}B|^{-1}2^{-2(i+2)}.
		\end{split}
	\end{equation}
	Combining \eqref{eq5.77} and \eqref{eq5.80},  we obtain (\ref{eqn:HC21_2}) and finish the proof of the claim.

    Plugging (\ref{eqn:HC21_2}) into (\ref{eq5.74}) and then into (\ref{eq5.73}), we arrive at 
	\begin{equation}
    \label{eqn:HC21_6}
		\begin{split}
			&\quad\int_{M\setminus 2B}\bigg|\int_0^\infty t^{-\frac{1}{2}}\int_{B} 
            \mathcal{K}(x,y,t)f(y)dydt\bigg|^p dx \\
			& \leq C\sum_{i=1}^{\infty} \int_{B^i} \left( |2^{i+1}B|^{-1}2^{-2i} \int_B |f| dy \right)^p dx \\
			&\leq C \left\{ \sum_{i=1}^{\infty} \int_{B^i} |2^{i+1}B|^{-p} 2^{-2ip} dx \right\} 
            \cdot \left\{ \int_B |f| dy \right\}^{p}. 
		\end{split}
	\end{equation}
    Recall that
    \begin{align*}
        \int_{B^i} |2^{i+1}B|^{-p} 2^{-2ip} dx&=\int_{2^{i+1}B \backslash 2^i B} |2^{i+1}B|^{-p} 2^{-2ip} dx\\
        &=2^{-2ip}\frac{|2^{i+1}B \backslash 2^i B|}{|2^{i+1}B|^p} \leq 2^{-2ip} |2^{i+1}B|^{1-p}. 
    \end{align*}
    The H\"older inequality yields that
    \begin{align*}
        \left\{ \int_B |f| dy \right\}^{p} \leq \left\{ \int_B |f|^p dy \right\} |B|^{p-1}. 
    \end{align*}
    Combining the previous two inequalities with (\ref{eqn:HC21_6}), we obtain 
    \begin{align*}
        &\quad \int_{M\setminus 2B}\bigg|\int_0^\infty t^{-\frac{1}{2}}\int_{B} 
            \mathcal{K}(x,y,t)f(y)dydt\bigg|^p dx\\
        &\leq C \left\{ \sum_{i=1}^{\infty} 2^{-2ip}  \left(\frac{|B|}{|2^{i+1}B|} \right)^{p-1} \right\} 
            \cdot \left\{ \int_B |f|^p dy \right\}\\
        &\leq  C \left\{ \sum_{i=1}^{\infty} 2^{-2ip} \right\} 
            \cdot \left\{ \int_B |f|^p dy \right\}\\
        &\leq C \int_B |f|^p dy, 
    \end{align*}
    which is exactly (\ref{eqn:HC18_4}). 
	\end{proof}
	
	\section{Proof of Theorem \ref{thm1.4}}
	We first introduce the well-known lemma (cf.\cite{BDG23,Str83}).
	\begin{lemma}\label{lm6.1}
	Assume that $\|Rm\|\leq \Lambda_0$. Then there is a $\kappa_0=\kappa_0(n,\Lambda_0)>0$  such that 
	\begin{equation}
    \label{eqn:HC21_7}
		\begin{split}
			\|\nabla (\Delta_j+\kappa)^{-1/2}f\|_{L^2}\leq \|f\|_{L^2}
		\end{split}
	\end{equation}
	for any $\kappa \geq \kappa_0$ and all $f\in \Gamma_{C^\infty_c}(M,\Lambda^jT^*M)$.
	\end{lemma}
	\begin{proof}
		By \eqref{eq2-10}, $\Delta_j=\nabla^*\nabla +V_j$. 
        For any $h \in W_2^1(M,\Lambda^jT^*M)$, 
        we have
		\begin{equation*}
			\begin{split}
				\|\nabla h\|^2_{L^2}&=\langle \nabla h,\nabla h\rangle
				 =\langle \nabla^*\nabla h, h\rangle
				 =\langle (\Delta_j+\kappa)h ,h\rangle-\langle (V_j+\kappa)h,h\rangle. 
			\end{split}
		\end{equation*}
		Since we choose $\kappa \geq \kappa_0 :=\|V_j\|_{L^\infty}$, the last term is nonnegative and can be dropped. Thus, we have
        \begin{align*}
            \|\nabla h\|^2_{L^2} \leq \langle (\Delta_j+\kappa)h ,h\rangle
            =\|(\Delta_j+\kappa)^{1/2}h\|^2_{L^2}.
        \end{align*}        
		Taking $h=(\Delta_j+\kappa )^{-1/2}f$, we arrive at (\ref{eqn:HC21_7}). 
	\end{proof}
    
	We are now ready to prove Theorem \ref{thm1.4}.
	\begin{theorem}\label{thm6.2}
		Let $M$ be an $n$-manifold with $\|Rm\|\leq \Lambda_0$. Then for each $p\in (1,\infty)$ there exists a $\kappa_0=\kappa_0(n,p,\Lambda_0)$ such that for every $\kappa\geq \kappa_0$ and $j \in \{0,\cdots,n\}$, we have
        \begin{align}
            \|\nabla (\Delta_j+\kappa)^{-\frac12}\|_{p,p}<C(n,p,\Lambda_0, \kappa)<\infty. 
        \label{eqn:HC21_8}    
        \end{align}
	\end{theorem}
	
	\begin{proof}
   We divide the proof into the following three steps.
    
     \textbf{Step 1: $p=2$.}
     
   By Lemma \ref{lm6.1}, we get the theorem for $p=2$.
   
    \textbf{Step 2: $1<p<2$.}

 In Theorem \ref{thm4-3}, we obtain the $C^1$ estimate \eqref{eqA11} for the heat kernel of $\Delta_j$. Following the methodology in \cite[Theorem 1.2]{CD99}, which was also adopted in \cite[Corollary 1.4]{BDG23}, one can show  that  $\nabla (\Delta_j+\kappa)^{-1/2}$ is of weak $(1,1)$ for $\kappa$ large enough, i.e.,
 \begin{equation*}
     \mu(\{\nabla (\Delta_j+\kappa)^{-1/2}f|>\lambda\})\leq \frac{C}{\lambda}\|f\|_{L^1}
 \end{equation*}
 for some $C=C(n,\Lambda_0,\kappa)>0$.
 By Marcinkiewicz interpolation theorem and Lemma \ref{lm6.1}, the theorem holds for $1<p<2$. We omit the details in this paper.
    
     \textbf{Step 3: $p>2$.}
     
		Let $H_j(x,y,t)$ be the heat kernel of $\Delta_j$ on $M$. If $\kappa$ is large enough($\kappa>C(n)(\Lambda_0+1)$ for example), by Theorem \ref{thm4-3} and Lemma \ref{lm6.1}, the operator
		\begin{equation*}
			\begin{split}
				Tf(x):=  \nabla (\Delta_j+\kappa)^{-\frac12}f(x) =\int_0^\infty t^{-\frac12}\int_M e^{-\kappa t}\nabla_x H_j(x,y,t)f(y)dydt
			\end{split}
		\end{equation*}
		is a Riesz type operator. 
		
		Since $M$ has bounded Ricci curvature, there exists $\{x_i\}_{i\in \IN}\subset M$ such that the balls $B_{1}(x_i)$  cover $M$, while the balls $B_{\frac{1}{2}}(x_i)$ are disjoint. Moreover, $B_1(x_i)$ has the finite intersection property. Let $\phi_i$ be a partition of unity subordinate to the covering $\{B_1(x_i)\}_{i=1}^{\infty}$. 
        Fix $f$ and set $f_i :=\phi_i f$. By the finite intersection property, there is a constant $C>0$ such that 
		\begin{equation*}
			C^{-1}\|f\|_{L^p}\leq \sum_i\|f_i\|_{L^p}\leq C\|f\|_{L^p}.
		\end{equation*}
		By Theorem \ref{thm5.2}, we have 
		\begin{equation*}
			\|Tf_i\|_{L^p}\leq  C\|f_i\|_{L^p}.
		\end{equation*}
		Summing them up, we obtain
		\begin{equation*}
			\|Tf\|_{L^p}\leq \sum_i\|Tf_i\|_{L^p}\leq C\sum_i\|f_i\|_{L^p}\leq C\|f\|_{L^p}.
		\end{equation*}
        By the arbitrary choice of $f$, this finishes the proof of (\ref{eqn:HC21_8}). 
	\end{proof}

\end{document}